\documentclass[11pt]{amsart}

\usepackage{amssymb,caption,longtable,mathtools,tikz}
\usepackage[colorlinks=true,allcolors=blue]{hyperref}
\usepackage[margin=1.25in]{geometry}

\usetikzlibrary{graphs,graphs.standard}

\makeatletter
\g@addto@macro\bfseries{\boldmath}
\makeatother

\newtheorem{theorem}{Theorem}
\newtheorem{conjecture}[theorem]{Conjecture}
\newtheorem{corollary}[theorem]{Corollary}
\newtheorem{lemma}[theorem]{Lemma}
\newtheorem{proposition}[theorem]{Proposition}
\newtheorem{question}[theorem]{Question}

\newcommand{\thistheoremname}{}
\newtheorem*{genericthm*}{\thistheoremname}
\newenvironment{namedthm*}[1]
  {\renewcommand{\thistheoremname}{#1}%
   \begin{genericthm*}}
  {\end{genericthm*}}

\theoremstyle{remark}
\newtheorem{example}[theorem]{Example}
\newtheorem{remark}[theorem]{Remark}

\makeatletter
\newcommand{\vast}{\bBigg@{3}}
\newcommand{\Vast}{\bBigg@{4}}
\makeatother

\newcommand{\vastl}{\mathopen\vast}
\newcommand{\vastr}{\mathclose\vast}
\newcommand{\Vastl}{\mathopen\Vast}
\newcommand{\Vastr}{\mathclose\Vast}

\DeclareMathOperator{\diag}{diag}
\DeclareMathOperator{\im}{im}
\DeclareMathOperator{\lcm}{lcm}
\DeclareMathOperator{\rank}{rank}

\DeclarePairedDelimiter{\abs}{\lvert}{\rvert}

\newcommand{\CD}{\mathcal{D}}
\newcommand{\CG}{\mathcal{CG}}
\newcommand{\CK}{\mathcal{K}}
\newcommand{\vd}{\mathbf{d}}
\newcommand{\vdhat}{\widehat{\vd}}
\newcommand{\vr}{\mathbf{r}}
\newcommand{\ZZ}{\mathbb{Z}}

\begin{document}

\title[Critical groups on star graphs and complete graphs]{Critical groups of arithmetical structures\\ on star graphs and complete graphs}

\author{Kassie Archer}
\address[K.~Archer]{Department of Mathematics\\ United States Naval Academy\\ Annapolis, MD 21402\\ USA}
\email{karcher@usna.edu}

\author{Alexander Diaz-Lopez}
\address[A.~Diaz-Lopez]{Department of Mathematics and Statistics\\ Villanova University\\ 800 Lancaster Ave (SAC 305)\\ Villanova, PA 19085\\ USA}
\email{alexander.diaz-lopez@villanova.edu}

\author{Darren Glass}
\address[D.~Glass]{Department of Mathematics\\ Gettysburg College\\ 300 N. Washington St\\ Gettysburg, PA 17325\\ USA}
\email{dglass@gettysburg.edu}

\author{Joel Louwsma}
\address[J.~Louwsma]{Department of Mathematics\\ Niagara University\\ Niagara University, NY 14109\\ USA}
\email{jlouwsma@niagara.edu}

\begin{abstract}
An arithmetical structure on a finite, connected graph without loops is an assignment of positive integers to the vertices that satisfies certain conditions. Associated to each of these is a finite abelian group known as its critical group. We show how to determine the critical group of an arithmetical structure on a star graph or complete graph in terms of the entries of the arithmetical structure. We use this to investigate which finite abelian groups can occur as critical groups of arithmetical structures on these graphs.
\end{abstract}

\maketitle

\section{Introduction}

An \emph{arithmetical structure} on a finite, connected, loopless graph~$G$ with vertex set $V(G)$ is a labeling of the vertices with positive integers $\vr=(r_v)_{v\in V(G)}$ and nonnegative integers $\vd=(d_v)_{v\in V(G)}$ such that, for all vertices~$v$, we have $r_vd_v=\sum r_u$, where the sum is taken over all~$u$ adjacent to~$v$, with multiplicity in the case of non-simple graphs, and where the entries of~$\vr$ have no common factor. The \emph{critical group} of an arithmetical structure is the torsion part of the cokernel of $L(G,\vd)\coloneqq\diag(\vd)-A(G)$, where $A(G)$ is the adjacency matrix of~$G$ and $\diag(\vd)$ is the diagonal matrix with the vector~$\vd$ on the diagonal, so that $\vr$ generates the kernel of $L(G,\vd)$. In this paper, we describe how to compute critical groups of arithmetical structures on star graphs and complete graphs and partially characterize the abelian groups that occur as critical groups of arithmetical structures on these graphs. 

Arithmetical structures were introduced by Dino Lorenzini~\cite{L89} to study intersections of degenerating curves in algebraic geometry. A number of papers in recent years (for example, \cite{A20, B18, CV, GW}) have studied arithmetical structures and their critical groups on various families of graphs, including path graphs, cycle graphs, bidents, and paths with a doubled edge. In particular, it has been shown that the critical groups associated to path graphs are trivial~\cite{B18} and those associated to cycle graphs~\cite{B18} or bidents~\cite{A20} are cyclic. In~\cite{L24}, Lorenzini showed that every finite, connected graph admits an arithmetical structure with trivial critical group. 

Let $S_n$ denote the star graph with $n$~leaves labeled $v_1,v_2,\dotsc, v_n$ connected to one central vertex labeled~$v_0$. For simplicity, we use $d_i$ as a shorthand for~$d_{v_i}$ and $r_i$ as a shorthand for~$r_{v_i}$. An example of an arithmetical structure on~$S_6$ is shown in Figure~\ref{fig:stargraphs}. As noted by Corrales and Valencia~\cite{CV}, arithmetical structures on~$S_n$ are in bijection with solutions in the positive integers of the equation
\begin{equation}\label{eq:diophantine}
\sum_{i=1}^n\frac{1}{d_i}=d_0.
\end{equation}
To see this, observe that, given an arithmetical structure $(\vd,\vr)$ on~$S_n$, we have $d_ir_i=r_0$ for all $i \in [n]$ and $d_0r_0=\sum_{i=1}^n r_i$. Hence,
\[\sum_{i=1}^n\frac{1}{d_i} = \sum_{i=1}^n\frac{r_i}{d_ir_i} = \sum_{i=1}^n\frac{r_i}{r_0} = d_0.\]
In the reverse direction, given a solution of~\eqref{eq:diophantine}, setting $r_0=\lcm(d_1, d_2, \dotsc, d_n)$ and $r_i=r_0/d_i$ for all $i \in [n]$ gives an arithmetical structure on~$S_n$. Solutions of~\eqref{eq:diophantine} have been much studied, sometimes under the name \emph{Egyptian fractions} with the assumption that the $d_i$'s are distinct, but many open questions about them remain. See, for example, \cite{EG, Graham, Guy} and the references therein.

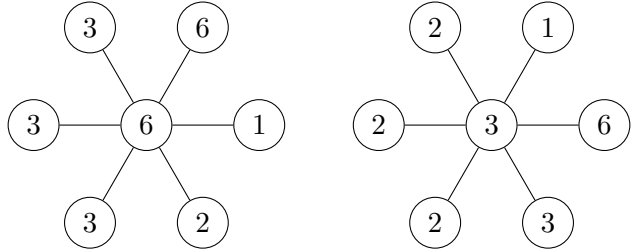
\begin{figure}
\centering
\begin{tikzpicture}
    \node[circle,draw,fill=white] at (0,0) (center) {$6$};
    \foreach \n in {1,...,1}{
        \node[circle,draw,fill=white] at ({\n*360/6}:1.5cm) (n\n) {$6$};
        \draw (center)--(n\n);}
    \foreach \n in {2,...,4}{
        \node[circle,draw,fill=white] at ({\n*360/6}:1.5cm) (n\n) {$3$};
        \draw (center)--(n\n);}
        \foreach \n in {5,...,5}{
        \node[circle,draw,fill=white] at ({\n*360/6}:1.5cm) (n\n) {$2$};
        \draw (center)--(n\n);}
        \foreach \n in {6,...,6}{
        \node[circle,draw,fill=white] at ({\n*360/6}:1.5cm) (n\n) {$1$};
        \draw (center)--(n\n);}
\end{tikzpicture}
\qquad
\begin{tikzpicture}
    \node[circle,draw,fill=white] at (0,0) (center) {$3$};
    \foreach \n in {1,...,1}{
        \node[circle,draw,fill=white] at ({\n*360/6}:1.5cm) (n\n) {$1$};
        \draw (center)--(n\n);}
    \foreach \n in {2,...,4}{
        \node[circle,draw,fill=white] at ({\n*360/6}:1.5cm) (n\n) {$2$};
        \draw (center)--(n\n);}
        \foreach \n in {5,...,5}{
        \node[circle,draw,fill=white] at ({\n*360/6}:1.5cm) (n\n) {$3$};
        \draw (center)--(n\n);}
        \foreach \n in {6,...,6}{
        \node[circle,draw,fill=white] at ({\n*360/6}:1.5cm) (n\n) {$6$};
        \draw (center)--(n\n);}
\end{tikzpicture}
\caption{An arithmetical structure on the star graph~$S_6$. The left side is the $\vr$-labeling and the right side is the $\vd$-labeling. The associated solution of~\eqref{eq:diophantine} is $\frac{1}{1}+\frac{1}{2}+\frac{1}{2}+\frac{1}{2}+\frac{1}{3}+ \frac{1}{6}=3$.\label{fig:stargraphs}}
\end{figure}

Since a star graph is a tree, a formula of Lorenzini \cite[Corollary~2.5]{L89} gives the order of the critical group of an arithmetical structure on~$S_n$ in terms of its $\vr$-labeling as 
\[\frac{r_0^{n-2}}{\prod_{i=1}^n r_i}.\]
However, critical groups of arithmetical structures on star graphs are often noncyclic, so the order does not determine the critical group. Section~\ref{sec:main} proves our main result, which finds the critical group of an arithmetical structure on a star graph in terms of its $\vd$-labeling. 
\begin{namedthm*}{Theorem~\ref{thm:criticalgroupstar}}
If $(\vd,\vr)$ is an arithmetical structure on~$S_n$ for $n\geq1$ and $\mathcal{K}(S_n;\vd,\vr)$ is its critical group, then
\[\mathcal{K}(S_n;\vd,\vr)\oplus (\mathbb{Z}/r_0\mathbb{Z})^2 \cong \bigoplus_{i=1}^n \mathbb{Z}/d_i\mathbb{Z}, \]
where $r_0=\lcm(d_1, d_2, \dotsc, d_n)$.
\end{namedthm*}
As an immediate corollary of this theorem together with the clique-star operation (see Section~\ref{sec:clique-star}), we can also compute critical groups of arithmetical structures on complete graphs in terms of~$\vd$.

In Section~\ref{sec:groupsthatappear}, we use Theorem~\ref{thm:criticalgroupstar} to investigate which groups occur as critical groups of arithmetical structures on star graphs and complete graphs. In particular, we:
\begin{itemize}
    \item construct trivial critical groups associated to $S_n$ and~$K_n$ (Section~\ref{subsec:trivial});
    \item construct infinitely many cyclic critical groups associated to star graphs and complete graphs (Section~\ref{subsec:cyclic}); 
    \item construct critical groups associated to $S_n$ and~$K_n$ with anywhere between $0$ and $n-2$ invariant factors, and show that a conjecture of Erd\H{o}s would imply that every finite abelian group is realized as a critical group associated to a star graph and complete graph (Section~\ref{subsec:highrank});
    \item construct products of critical groups associated to star graphs by identifying center vertices of smaller star graphs (Section~\ref{subsec:products});
    \item show that every finite abelian group occurs as a subgroup of a critical group associated to $S_n$ and~$K_n$ for some~$n$ (Section~\ref{subsec:subgroup});
    \item bound the maximal order of a critical group associated to $S_n$ or~$K_n$ for a given~$n$ (Section~\ref{subsec:order}); and
    \item investigate how many distinct groups appear as critical groups associated to~$S_n$ (Section~\ref{subsec:number}).
\end{itemize}

\section{Preliminaries}

\subsection{Smith normal forms and invariant factors}\label{subsec:smith}

The \emph{critical group} $\mathcal{K}(G;\vd,\vr)$ of a given arithmetical structure $(\vd,\vr)$ on a graph~$G$ is defined to be the torsion part of the cokernel of the matrix $L(G,\vd)=\diag(\vd)-A(G)$. To compute this group, we recall the following facts about the Smith normal form of a matrix. For more details, we refer the reader to \cite{GK, L08, Stanley}.

Given an $n\times n$ matrix~$M$ with integer entries, there exist invertible $n\times n$ matrices $S$ and~$T$, both with integer entries, such that the product $SMT$ is of the form
\[\scriptsize{\begin{bmatrix}
\alpha_1 & 0 & \cdots & \cdots & \cdots & \cdots & 0\\
0 & \alpha_2 & \ddots & & & & \vdots\\
\vdots & \ddots & \ddots & \ddots & & & \vdots\\
\vdots & & \ddots & \alpha_{t} & \ddots & & \vdots\\
\vdots & & & \ddots & 0 & \ddots & \vdots\\
\vdots & & & & \ddots & \ddots & 0\\
0 & \cdots & \cdots & \cdots & \cdots & 0 & 0
\end{bmatrix},}\]
where $t=\rank(M)$, each~$\alpha_k$ is a positive integer, and $\alpha_k$ divides $\alpha_{k+1}$ for all $k\in[t-1]$. This diagonal matrix is the \emph{Smith normal form} of~$M$, and the entries~$\alpha_k$ are the \emph{invariant factors} of~$M$. Thinking of~$M$ as a linear map $\mathbb{Z}^n \to \mathbb{Z}^n$, the cokernel $\mathbb{Z}^n/\im(M)$ is isomorphic to 
\[\mathbb{Z}^{n-t}\oplus\Biggl(\bigoplus_{k=1}^{t} \mathbb{Z}/\alpha_{k}\mathbb{Z}\Biggr).\]
Letting $D_{k}(M)$ be the greatest common divisor of all $k\times k$ minors of~$M$ and defining $D_0(M)=1$, one has that $\alpha_{k}=D_{k}(M)/D_{k-1}(M)$ for all $k\in[t]$. 

In the case of $L(G,\vd)$, the matrix $(\diag(\vd)-A(G))$ associated to an arithmetical structure $(\vd,\vr)$ on a graph~$G$ with $n$~vertices and adjacency matrix $A(G)$, Lorenzini \cite[Proposition~1.1]{L89} showed that $\rank(L(G,\vd))=n-1$. Since the nontorsion part of the cokernel is always~$\ZZ$, to find the cokernel we only need to compute the torsion part. The critical group of $(\vd,\vr)$ is defined to be this torsion part, namely
\[\mathcal{K}(G;\vd,\vr)\cong\bigoplus_{k=1}^{n-1} \mathbb{Z}/\alpha_{k}\mathbb{Z}.\]

\subsection{Clique-star operation}\label{sec:clique-star} 

Corrales and Valencia~\cite{CV} defined a clique-star operation on graphs with arithmetical structures that preserves the critical group and is a special case of the blowup operation described by Lorenzini~\cite{L89}. This operation replaces a clique subgraph of a graph with an arithmetical structure by a star subgraph to produce an arithmetical structure on a graph with one more vertex. Keyes and Reiter~\cite{KR} later defined an operation that generalizes the inverse of this clique-star operation, and~\cite{DL} studied how this generalized star-clique operation transforms critical groups. 

We describe the clique-star operation in the special case of the complete graph~$K_n$ with vertices $v_1,v_2,\dotsc,v_n$ and the star graph~$S_n$ with leaves $v_1,v_2,\dotsc,v_n$ connected to a central vertex~$v_0$. Suppose $\vd=(d_1,d_2,\dotsc, d_n)$ and $\vr=(r_1,r_2,\dotsc, r_n)$ give an arithmetical structure on~$K_n$. Then there is a corresponding arithmetical structure on~$S_n$ given by 
\begin{align*}
\vd'&=(d_1',d_2' \dotsc, d_n',d_0') = (d_1+1, d_2+1, \dotsc, d_n+1,1)\\
\vr'&=(r_1',r_2' \dotsc, r_n',r_0') = \Bigl(r_1, r_2, \dotsc, r_n, \sum r_i\Bigr).
\end{align*}
An example of this operation is shown in Figure~\ref{fig:cliquestar}.

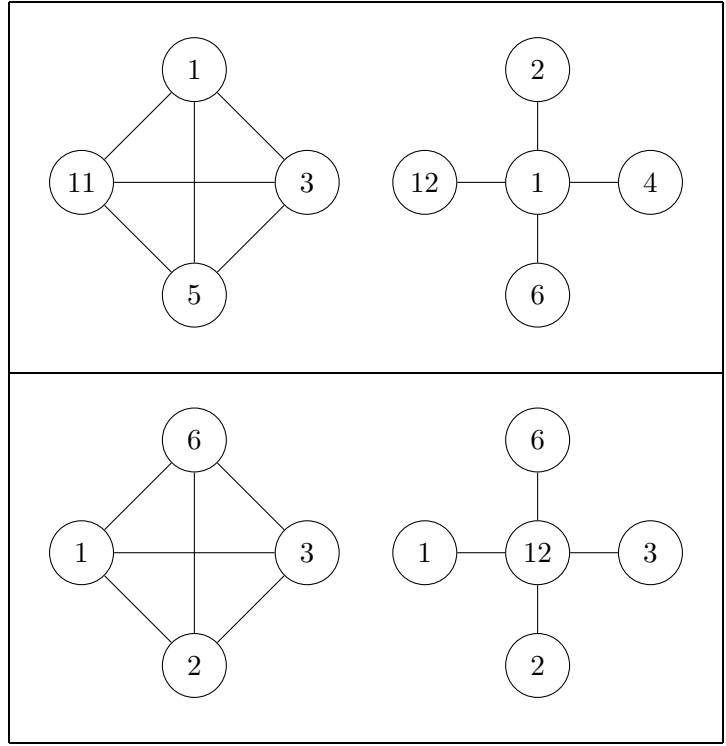
\begin{figure}
\centering
\begin{tabular}{|ccccc|}
\hline
&&&&\\
&\begin{tikzpicture} 
  \graph[nodes={draw, circle, minimum size=.85cm}, n=4, clockwise, radius=1.5cm] { subgraph K_n [V={1,3,5,11}];};
\end{tikzpicture} & &
\begin{tikzpicture}
    \graph[nodes={draw, circle, minimum size=.85cm},n=4] { subgraph I_n [V={2,4,6,12}, clockwise,radius=1.5cm] -- 1 };
\end{tikzpicture}&\\
&&&&\\\hline
&&&&\\
&\begin{tikzpicture} 
  \graph[nodes={draw, circle, minimum size=.85cm}, n=4, clockwise, radius=1.5cm] { subgraph K_n [V={6,3,2,1}];};
\end{tikzpicture} &&
\begin{tikzpicture}
    \graph[nodes={draw, circle, minimum size=.85cm},n=4] { subgraph I_n [V={6,3,2,1}, clockwise,radius=1.5cm] -- 12 };
\end{tikzpicture} &\\
&&&& \\ \hline
\end{tabular}
\caption{An example of the clique-star operation, where the $\vd$-values are given in the upper half of the figure and the $\vr$-values are given in the lower half. \label{fig:cliquestar}}
\end{figure}

It is straightforward to check that $(\vd',\vr')$ is an arithmetical structure on~$S_n$, and the clique-star operation gives a bijection between arithmetical structures on~$S_n$ with $d_0=1$ and arithmetical structures on~$K_n$ \cite[Theorem~5.1]{CV}. Moreover, the critical group of $(\vd,\vr)$ on~$K_n$ is isomorphic to the critical group of $(\vd',\vr')$ on~$S_n$ \cite[Theorem~5.3]{CV}. Therefore the groups that occur as critical groups of arithmetical structures on the complete graph~$K_n$ are precisely those that occur as critical groups of arithmetical structures on the star graph~$S_n$ with $d_0=1$.

\section{Computing critical groups}\label{sec:main}

The primary purpose of this section is to prove our main theorem about computing critical groups of arithmetical structures on star graphs in terms of their $\vd$-values.

\begin{theorem}\label{thm:criticalgroupstar}
If $(\vd,\vr)$ is an arithmetical structure on~$S_n$ for $n\geq1$ and $\mathcal{K}(S_n;\vd,\vr)$ is its critical group, then
\[\mathcal{K}(S_n;\vd,\vr)\oplus (\mathbb{Z}/r_0\mathbb{Z})^2 \cong \bigoplus_{i=1}^n \mathbb{Z}/d_i\mathbb{Z},\]
where $r_0=\lcm(d_1, d_2, \dotsc, d_n)$.
\end{theorem}

In order to prove this theorem, we will show that the matrix 
\[B(\vd)\coloneqq L(S_n,\vd) \oplus \begin{bmatrix} r_0 &0 \\ 0 & r_0\end{bmatrix}\]
and the diagonal matrix $C(\vd) \coloneqq\diag(d_1, d_2, \dotsc, d_n, 1, 1, 0)$ have the same Smith normal form. It is enough to show that the greatest common divisors of all $k\times k$ minors of each, which, following Section~\ref{subsec:smith}, we denote $D_k(B(\vd))$ and $D_k(C(\vd))$, respectively, are equal to each other. We in fact show that $D_k(B(\vd))$ and $D_k(C(\vd))$ are both equal to the greatest common divisor of all products of $k-2$ entries of $(d_1, d_2, \dotsc, d_n)$. To this end, we first state a few lemmas.

\begin{lemma}\label{lem:primes}
Let $n\geq2$. Suppose $d_1,d_2 \dotsc, d_n, d_0$ are positive integers with $\sum_{i=1}^n\frac{1}{d_i}=d_0$. If $p$ is a prime and $a$ is the largest integer for which $p^a$ divides $\lcm(d_1,d_2,\dotsc,d_n)$, then there exists $\{i,j\}\subseteq [n]$ such that $p^a$ divides both $d_i$ and~$d_j$.
\end{lemma}

\begin{proof}
If $a=0$ the result is trivially true, so suppose $a\geq 1$. Since the largest power of~$p$ that divides $\lcm(d_1,d_2,\dotsc,d_n)$ is the maximum of the largest powers of~$p$ that divide the $d_i$'s, there is some $i\in[n]$ for which $p^a$ divides~$d_i$. Let $c\coloneqq\lcm(d_1,d_2,\dotsc, d_n)$. Multiplying the equation $\sum_{i=1}^n\frac{1}{d_i}=d_0$ through by $c$ and separating out the $i$-th term, we obtain
\begin{equation}\label{eq:primes}
\frac{c}{d_i}+\sum_{\substack{j=1\\j\neq i}}^n \frac{c}{d_j}=c\cdot d_0.
\end{equation}
Since $p^a$ divides $c$, it also divides the right side of~\eqref{eq:primes}. Analyzing the left side, since the highest powers of~$p$ that divide $c$ and~$d_i$ are the same, we have that $p$ does not divide $c/d_i$. Therefore there must be some $j\in[n]\setminus\{i\}$ such that ${c}/{d_j}$ is not divisible by~$p$. This exactly implies that $p^a$ divides~$d_j$.
\end{proof}

\begin{lemma}\label{prop:stardet}
For any integers $d_1,d_2\dotsc, d_n,d_0$, we have that 
\[\scriptsize{\begin{vmatrix}
    d_1 & 0 & \cdots & 0 & -1 \\
    0 & d_2 & \ddots & \vdots & -1 \\
    \vdots & \ddots & \ddots & 0 & \vdots \\
    0 & \cdots & 0 & d_n & -1 \\
    -1 & -1 & \cdots & -1 & d_0 \\
\end{vmatrix}}=\prod_{i=0}^n d_i - \sum_{i=1}^{n}\prod_{\substack{j\in[n]\\j\neq i}} d_j.\]
\end{lemma}

\begin{proof}
We proceed by induction on~$n$. In the base case $n=0$, both sides equal~$d_0$. For the inductive step, we expand along row~$n$ to get that
\[\begin{split}
\scriptsize{
\begin{vmatrix}
    d_1 & 0 & \cdots & 0 & -1 \\
    0 & d_2 & \ddots & \vdots & -1 \\
    \vdots & \ddots & \ddots & 0 & \vdots \\
    0 & \cdots & 0 & d_n & -1 \\
    -1 & -1 & \cdots & -1 & d_0 \\
\end{vmatrix}}&=d_n\scriptsize{\begin{vmatrix}
    d_1 & 0 & \cdots & 0 & -1 \\
    0 & d_2 & \ddots & \vdots & -1 \\
    \vdots & \ddots & \ddots & 0 & \vdots \\
    0 & \cdots & 0 & d_{n-1} & -1 \\
    -1 & -1 & \cdots & -1 & d_0 \\
\end{vmatrix}}+\scriptsize{\begin{vmatrix}
    d_1 & 0 & 0 & \cdots & 0 \\
    0 & d_2 & \ddots & \ddots & \vdots \\
    \vdots & \ddots & \ddots & \ddots & 0 \\
    0 & \cdots & 0 & d_{n-1} & 0 \\
    -1 & -1 & -1 & \cdots & -1 \\
\end{vmatrix}}\\
&=d_n\vastl(\prod_{i=0}^{n-1} d_i - \sum_{i=1}^{n-1}\prod_{\substack{j\in[n-1]\\j\neq i}} d_j\vastr)-\scriptsize{\begin{vmatrix}
    d_1 & 0 & \cdots & 0 \\
    0 & d_2 & \ddots & \vdots \\
    \vdots & \ddots & \ddots & 0 \\
    0 & \cdots & 0 & d_{n-1} \\
\end{vmatrix}}\\
&=\prod_{i=0}^n d_i - \sum_{i=1}^{n}\prod_{\substack{j\in[n]\\j\neq i}} d_j.\qedhere
\end{split}\]
\end{proof}

Given a vector $\vd=(d_1,d_2 \dotsc, d_n,d_0)$ such that $\sum_{i=1}^n\frac{1}{d_i}=d_0$, we create the vector $\vdhat\coloneqq(d_1,d_2,\dotsc, d_n)$, obtained by removing the $d_0$~entry from~$\vd$. Let $G_k\bigl(\vdhat\bigr)$ be the greatest common divisor of all products of $k-2$ entries of~$\vdhat$, i.e.\ those products of the form $\prod_{j=1}^{k-2} d_{i_j}$ where $\{i_1, i_2, \dotsc, i_{k-2}\} \subseteq [n]$. The next two lemmas compute $D_k(B(\vd))$, the greatest common divisor of all $k\times k$ minors of the matrix 
\[ B(\vd)= \scriptsize{\begin{bmatrix}
    d_1 & 0 & \cdots & 0 & -1 & 0 & 0 \\
    0 & d_2 & \ddots & \vdots & -1 & \vdots & \vdots \\
    \vdots & \ddots & \ddots & 0 & \vdots & \vdots & \vdots \\
    0 & \cdots & 0 & d_n & -1 & \vdots & \vdots \\
    -1 & -1 & \cdots & -1 & d_0 & 0 & \vdots \\
    0 & \cdots & \cdots & \cdots & 0 & r_0 & 0 \\
    0 & \cdots & \cdots & \cdots & \cdots & 0 & r_0 \\
\end{bmatrix}},\]
where $r_0=\lcm(d_1,d_2,\dotsc, d_n)$. We find that $D_k(B(\vd))$ is exactly $G_k\bigl(\vdhat\bigr)$ by showing first that $D_k(B(\vd))$ divides $G_k\bigl(\vdhat\bigr)$ and then that $G_k\bigl(\vdhat\bigr)$ divides $D_k(B(\vd))$. 

\begin{lemma}\label{lem:minorsofB1}
Let $n\geq2$. Suppose $d_1,d_2,\dotsc,d_n,d_0$ are positive integers with $\sum_{i=1}^n\frac{1}{d_i}=d_0$. For all $k\in\{2,3,\dotsc,n+2\}$, we have that $D_k(B(\vd))$ divides $G_k\bigl(\vdhat\bigr)$.
\end{lemma}

\begin{proof}
We will show, for any subset $I = \{i_1,i_2,\dotsc,i_{k-2}\}\subseteq[n]$, that $D_k(B(\vd))$ divides the product $\prod_{i\in I}d_{i}$. Using B\'{e}zout's identity, this would imply that $D_k(B(\vd))$ divides $G_k\bigl(\vdhat\bigr)$, the greatest common divisor of all such products.

First, let $k\in\{2,3,\dotsc, n\}$ and consider a subset $I \subseteq[n]$ of cardinality $k-2$. Let $\{\ell,m\}\subseteq[n]\setminus I$, and consider the $k\times k$ submatrix of $B(\vd)$ determined by rows indexed by $I \cup\{\ell,n+1\}$ and columns indexed by $I \cup\{m,n+1\}$. Computing the determinant of this submatrix by expanding along the row originally indexed by~$\ell$ and then the column originally indexed by~$m$, we get that the corresponding minor equals $\pm\prod_{i \in I}d_{i}$. Therefore, $D_k(B(\vd))$ divides $\prod_{i \in I}d_{i}$. 

For the cases $k\in \{n+1,n+2\}$, we will show that, for any prime~$p$ and any subset $I\subseteq[n]$ of cardinality $k-2$, the largest power of~$p$ that divides $D_k(B(\vd))$ must also divide $\prod_{i\in I} d_i$. 

In the case $k=n+1$, a subset of $[n]$ of cardinality $k-2$ must be of the form $I = [n]\setminus\{m\}$ for some $m\in[n]$. Let $p$ be prime, let $a$ be the largest integer for which $p^a$ divides $r_0$, and let $a_i$ be the largest integer for which $p^{a_i}$ divides~$d_i$. By Lemma~\ref{lem:primes}, there is some $s\in I$ for which $a_s=a$. Consider the $(n+1)\times(n+1)$ submatrix of $B(\vd)$ determined by the rows in $I\cup\{n+1, n+2\}$ and the columns in $[n]\cup\{n+1,n+2\}\setminus\{s\}$. Computing the determinant of this submatrix by expanding along the row originally indexed by~$s$ and then the column originally indexed by~$m$, we get that the corresponding minor is $\pm r_0 \prod_{i \in I\setminus\{s\}} d_i$. Since $D_{n+1}(B(\vd))$ divides this minor, we have that the largest power of~$p$ that divides $D_{n+1}(B(\vd))$ is at most 
\[a+\sum_{\substack{i \in I\\i\neq s}} a_i = \sum_{i\in I} a_i,\]
which is exactly the largest power of~$p$ that divides the product $\prod_{i\in I} d_i$. Since this argument holds for any prime, we conclude that $D_{n+1}(B(\vd))$ divides $\prod_{i\in I} d_i$.

The case $k=n+2$ is similar. In this case, the only subset of $[n]$ with cardinality $k-2$ is $[n]$ itself. For a prime~$p$, define $a$ and~$a_i$ as above. By Lemma~\ref{lem:primes}, there is $\{s,t\}\subseteq [n]$ such that $a_s=a_t=a$. Consider the $(n+2)\times(n+2)$ submatrix of $B(\vd)$ determined by the rows indexed by $[n+3]\setminus\{s\}$ and the columns indexed by $[n+3]\setminus\{t\}$. Computing the determinant of this submatrix by expanding along the row originally indexed by~$t$ and then the column originally indexed by~$s$, we find the determinant to be 
\[\pm r_0^2\prod_{\substack{i \in [n]\\i\neq s,t}}d_{i}.\]
Since $D_{n+2}(B(\vd))$ divides this minor, the largest power of~$p$ dividing $D_{n+2}(B(\vd))$ is at most 
\[2a+\sum_{\substack{i \in [n]\\i\neq s,t}}a_{i} = \sum_{i=1}^{n}a_{i},\]
which is the largest power of~$p$ dividing $\prod_{i=1}^{n}d_{i}$. Since this argument works for any prime, we conclude that $D_{n+2}(B(\vd))$ divides $\prod_{i=1}^{n}d_{i}$. 
\end{proof}

\begin{lemma}\label{lem:minorsofB2}
Suppose $d_1,d_2\dotsc,d_n,d_0$ are positive integers with $\sum_{i=1}^n\frac{1}{d_i}=d_0$. For all $k\in\{2,3,\dotsc,n+2\}$, we have that $G_k\bigl(\vdhat\bigr)$ divides $D_k(B(\vd))$.
\end{lemma}

\begin{proof}
It is enough to show that all $k\times k$ minors of $B(\vd)$ are divisible by $G_k\bigl(\vdhat\bigr)$, the greatest common divisor of all products of $k-2$ entries of~$\vdhat$, as the result then follows by using B\'{e}zout's identity.

First, consider principal minors of $B(\vd)$. If a principal submatrix does not include row and column $n+1$, the corresponding minor is the determinant of a diagonal matrix with $k$~diagonal entries. At least $k-2$ of these diagonal entries must be entries of~$\vdhat$, say $d_{i_1},d_{i_2}, \dotsc, d_{i_{k-2}}$, so we have that the minor is divisible by $\prod_{j=1}^{k-2} d_{i_j}$ and thus by $G_k\bigl(\vdhat\bigr)$. 

If the rows and columns of a principal submatrix are indexed by $I\cup\{n+1\}$ for some subset $I\subseteq [n]$ of cardinality $k-1$, then by Lemma~\ref{prop:stardet} the associated minor is
\[d_0 \prod_{i \in I} d_i - \sum_{i \in I} \prod_{\substack{j\in I\\j\neq i}} d_j.\]
Since each term of the above sum has a product of at least $k-2$ entries of~$\vdhat$, each term is divisible by $G_k\bigl(\vdhat\bigr)$ and hence the minor is divisible by $G_k\bigl(\vdhat\bigr)$.

If the rows and columns of a principal submatrix are indexed by $I\cup\{n+1, n+2\}$ or $I\cup\{n+1,n+3\}$ for a subset $I\subseteq[n]$ of cardinality $k-2$, then using Lemma~\ref{prop:stardet} the associated minor is
\[d_0r_0 \prod_{i \in I} d_i - \sum_{i \in I}r_0 \prod_{\substack{j\in I\\j\neq i}} d_j.\]
Since each term of the above sum has $r_0$ times a product of at least $k-3$ entries of~$\vdhat$ and $r_0$ is divisible by~$d_i$ for every $i\in [n]$, each term is divisible by $G_k\bigl(\vdhat\bigr)$ and hence the minor is divisible by $G_k\bigl(\vdhat\bigr)$.

If the rows and columns of a principal submatrix are indexed by $I\cup\{n+1, n+2,n+3\}$ for a subset $I\subseteq[n]$ of cardinality $k-3$, then using Lemma~\ref{prop:stardet} the associated minor is
\[d_0r_0^2 \prod_{i \in I} d_i - \sum_{i \in I}r_0^2 \prod_{\substack{j\in I\\j\neq i}} d_j.\]
Since each term of the above sum has $r_0^2$ times a product of at least $k-4$ entries of~$\vdhat$ and $r_0$ is divisible by~$d_i$ for every $i\in [n]$, each term is divisible by $G_k\bigl(\vdhat\bigr)$ and hence the minor is divisible by $G_k\bigl(\vdhat\bigr)$.

Finally, consider non-principal minors. For a non-principal minor of $B(\vd)$ to be nonzero, the rows of the corresponding submatrix must be indexed by $I\cup\{s\}$ and the columns by $I\cup\{t\}$, where $I$ is a subset of $[n+3]$ of cardinality $k-1$ and $s$ and~$t$ are distinct elements of $[n+1]$ that are not in~$I$, with $n+1\in I\cup\{s,t\}$.

If $s=n+1$, computing the determinant of the submatrix by expanding along the column originally labeled by~$t$ gives that it is (up to sign) a product of $k-1$ entries of $(d_1, d_2, \dotsc, d_n, r_0, r_0)$. If at least $k-2$ of these are entries of~$\vdhat$, then the minor is divisible by a product of $k-2$ such entries and is thus divisible by $G_k\bigl(\vdhat\bigr)$. If the minor is of the form $\pm r_0^2\prod_{i\in J} d_i$ for a subset $J \subseteq[n]$ of cardinality $k-3$, then since $d_i$ divides $r_0$ for all $i\in[n]$ we have that the minor is divisible by some product of $k-2$ entries of~$\vdhat$ and hence by $G_k\bigl(\vdhat\bigr)$. When $t=n+1$, expanding along the row originally indexed by~$s$ gives the same result.

When $n+1\in I$, computing the determinant of the submatrix by expanding along the row originally indexed by~$s$ and the column originally indexed by~$t$ gives that the minor is (up to sign) a product of $k-2$ entries of $(d_1, d_2, \dotsc, d_n, r_0, r_0)$. As before, since $d_i$ divides $r_0$ for all $i\in[n]$, the minor is divisible by some product of $k-2$ entries of~$\vdhat$ and thus by $G_k\bigl(\vdhat\bigr)$. 
\end{proof}

Together, Lemmas \ref{lem:minorsofB1} and~\ref{lem:minorsofB2} show that $D_k(B(\vd)) = G_k\bigl(\vdhat\bigr)$ for $k\in\{2,3,\dotsc, n+2\}$. It remains to show the same is true for $D_k(C(\vd))$. 

\begin{lemma}\label{lem:minorsofC}
Let $d_1,d_2,\dotsc,d_n$ be integers. Consider the $(n+3)\times(n+3)$ matrix 
\[ C(\vd)=\diag(d_1,d_2\dotsc, d_n,1,1,0).\] 
If $k\in\{2,3,\dotsc,n+2\}$, then $D_k(C(\vd))=G_k\bigl(\vdhat\bigr)$.
\end{lemma}

\begin{proof}
Any non-principal minor will be zero, so we only need to consider principal minors. Each nonzero principal minor is a product of~$k$ entries of $(d_1,d_2, \dotsc, d_n,1,1)$. It is straightforward to see that the greatest common divisor of these products is exactly the greatest common divisor of the products of $k-2$ entries of~$\vdhat$. 
\end{proof}

We can now prove Theorem~\ref{thm:criticalgroupstar} using the preceding lemmas.

\begin{proof}[Proof of Theorem~\ref{thm:criticalgroupstar}]
If $n=1$ the result is trivially true, so assume $n\geq2$. Consider the $(n+3)\times (n+3)$ matrices
\[ B(\vd)=\scriptsize{\begin{bmatrix}
    d_1 & 0 & \cdots & 0 & -1 & 0 & 0 \\
    0 & d_2 & \ddots & \vdots & -1 & \vdots & \vdots \\
    \vdots & \ddots & \ddots & 0 & \vdots & \vdots & \vdots \\
    0 & \cdots & 0 & d_n &-1 & \vdots & \vdots \\
    -1 & -1 & \cdots & -1 & d_0 & 0 & \vdots \\
    0 & \cdots & \cdots & \cdots & 0 & r_0 & 0 \\
    0 & \cdots & \cdots & \cdots & \cdots & 0 & r_0 \\
\end{bmatrix}},\]
whose cokernel is isomorphic to $\mathbb{Z}\oplus \mathcal{K}(S_n;\vd,\vr)\oplus (\mathbb{Z}/r_0\mathbb{Z})^2$, and 
\[ C(\vd)=\diag(d_1,d_2,\dotsc, d_n,1,1,0),\]
whose cokernel is isomorphic to $\mathbb{Z}\oplus (\bigoplus_{i=1}^n\mathbb{Z}/d_i \mathbb{Z})$. To prove the result, it suffices to show that $B(\vd)$ and $C(\vd)$ have the same Smith normal form. This in turn follows from showing that $D_k(B(\vd))=D_k(C(\vd))$ for all $k\in \{0,1,\dotsc, n+3\}$. By definition, $D_0(B(\vd))=D_0(C(\vd))=1$ and $D_1(B(\vd))=D_1(C(\vd)) = 1$ as these matrices both have an entry of $\pm1$. Since the upper left $(n+1)\times (n+1)$ submatrix of $B(\vd)$ is the matrix of an arithmetical structure, it has determinant~$0$, and thus $D_{n+3}(B(\vd))=D_{n+3}(C(\vd))=0$. We therefore only need to consider the cases with $k\in\{2,3,\dotsc,n+2\}$. The result then follows from Lemmas \ref{lem:minorsofB1}, \ref{lem:minorsofB2}, and~\ref{lem:minorsofC}.
\end{proof}

As a result of Theorem~\ref{thm:criticalgroupstar}, it is a straightforward exercise to obtain the values of~$\alpha_{k}$ from~$\vdhat$. For example, if $\vdhat=(2,3,4,4,6,9,9,10,15,18,18)$, then $r_0=180$ and Theorem~\ref{thm:criticalgroupstar} implies that the critical group $\CK\coloneqq\mathcal{K}(S_n;\vd,\vr)$ satisfies
\[\CK\oplus(\ZZ/180\ZZ)^2 \cong \ZZ/2\ZZ\oplus\ZZ/3\ZZ \oplus (\ZZ/4\ZZ)^2\oplus \ZZ/6\ZZ \oplus (\ZZ/9\ZZ)^2 \oplus \ZZ/10\ZZ \oplus \ZZ/15\ZZ \oplus (\ZZ/18\ZZ)^2.\]
Notice that $180=2^2\cdot 3^2\cdot 5$. Rewriting the expression using primary decomposition we get
\[\CK\oplus(\ZZ/4\ZZ)^2\oplus(\ZZ/9\ZZ)^2\oplus(\ZZ/5\ZZ)^2 \cong (\ZZ/2\ZZ)^5 \oplus(\ZZ/4\ZZ)^2\oplus(\ZZ/3\ZZ)^3\oplus (\ZZ/9\ZZ)^4\oplus (\ZZ/5\ZZ)^2.\]
It is then clear that
\[\CK\cong(\ZZ/2\ZZ)^5\oplus(\ZZ/3\ZZ)^3\oplus(\ZZ/9\ZZ)^2 \cong(\ZZ/6\ZZ)^3\oplus(\ZZ/18\ZZ)^2.\]

In general, for each prime~$p$ and $i\in[n]$, let $a_{p,i}$ be the largest integer for which $p^{a_{p,i}}$ divides $d_i$, so that $d_i=\prod p^{a_{p,i}}$. For each~$p$, let $\{e_{p,1}\leq e_{p,2}\leq \dotsb \leq e_{p,n}\}$ be equal (as sets) to $\{a_{p,1},a_{p,2},\dotsc,a_{p,n}\}$, reindexed to appear in nondecreasing order. It then follows from Theorem~\ref{thm:criticalgroupstar} that $D_k(B(\vd))=\prod p^{\sum_{i=1}^{k-2}e_{p,i}}$. (When $k\in\{0,1,2\}$, the sum is empty and we have that $D_0(B(\vd))=D_1(B(\vd))=D_2(B(\vd))=1$.) Therefore, we have $\alpha_1=\alpha_2=1$ and, for all $k\in\{3,4,\dotsc,n\}$,
\[\alpha_k=\frac{D_k(B(\vd))}{D_{k-1}(B(\vd))}=\prod p^{e_{p,k-2}}.\]

\begin{proposition}\label{prop:CriticalGroupStarWithTwo1}
Let $(\vd,\vr)$ be an arithmetical structure on~$S_n$. If $r_i=r_j=1$ for some $i,j \in[n]$ with $i\neq j$, then 
\[\mathcal{K}(S_n;\vd,\vr)\cong \bigoplus_{\substack{k\in[n]\\ k\neq i,j}}\mathbb{Z}/d_k\mathbb{Z}.\]
\end{proposition} 

\begin{proof}
Theorem~\ref{thm:criticalgroupstar} gives that 
\[\mathcal{K}(S_n;\vd,\vr)\oplus (\mathbb{Z}/r_0\mathbb{Z})^2 \cong \bigoplus_{k=1}^n \mathbb{Z}/d_{k}\mathbb{Z}.\]
Since $r_i=r_j=1$, we have that $d_i=d_j=r_0$. Therefore we have
\[\mathcal{K}(S_n;\vd,\vr)\oplus (\mathbb{Z}/r_0\mathbb{Z})^2\cong \Vastl(\bigoplus_{\substack{k\in[n]\\ k\neq i,j}}\mathbb{Z}/d_k\mathbb{Z}\Vastr)\oplus (\mathbb{Z}/r_0\mathbb{Z})^2,\]
from which the result follows.
\end{proof}

An immediate corollary of Section~\ref{sec:clique-star} and Theorem~\ref{thm:criticalgroupstar} allows us to determine the critical groups of arithmetical structures on complete graphs.

\begin{corollary}\label{cor:criticalgroupcomplete}
If $(\vd,\vr)$ is an arithmetical structure on the complete graph~$K_n$, then
\[\mathcal{K}(K_n;\vd,\vr)\oplus (\mathbb{Z}/\ell\mathbb{Z})^2 \cong \bigoplus_{i=1}^n \mathbb{Z}/(d_i+1)\mathbb{Z},\]
where $\ell=\lcm(d_1+1,d_2+1,\dotsc, d_n+1)$.
\end{corollary}

\begin{proof}
We perform the clique-star operation on~$K_n$ with clique subgraph~$K_n$ to get an arithmetical structure $(\vd',\vr')$ on~$S_n$ with $\vd'=\vd+\textbf{1}$. By \cite[Theorem~5.3]{CV}, we have
\[\mathcal{K}(K_n;\vd,\vr)\cong \mathcal{K}(S_n;\vd',\vr').\] 
The result follows by applying Theorem~\ref{thm:criticalgroupstar}.
\end{proof}

\begin{remark}
Corollary~\ref{cor:criticalgroupcomplete} allows us to recover the well-known result (originally stated in \cite[Example~1.10]{L89}) that the critical group of the Laplacian arithmetical structure on~$K_n$ is $(\mathbb{Z}/n\mathbb{Z})^{n-2}$ as $\vd$ is the vector all of whose entries are $n-1$ and $r_0=d_i'=n$ for all $i\in[n]$.
\end{remark}

\begin{remark}
More generally, given any arithmetical structure $(\vd,\vr)$ on~$S_n$, we can apply the generalized star-clique operation given in~\cite{KR} to get an arithmetical structure $(\vd',\vr')$ on $d_0K_n$, the graph with $d_0$~edges between any two vertices. We have that $d'_i=d_0d_i$ for all $i\in[n]$, so Theorem~\ref{thm:criticalgroupstar} therefore gives that 
\[\mathcal{K}(d_0 K_n;\vd',\vr')\oplus(\ZZ/d_0r_0\ZZ)^2\cong\bigoplus_{i=1}^{n}\mathbb{Z}/d_0d_i\mathbb{Z}.\]
\end{remark}

\section{Groups that appear as critical groups}\label{sec:groupsthatappear}

This section is motivated by the following question.

\begin{question} \label{ques:allgroups}
Can every finite abelian group \textup{(}up to isomorphism\textup{)} be realized as the critical group of some arithmetical structure on a star graph? On a complete graph?
\end{question}

We obtain a partial answer to this question by considering certain families of groups in the following subsections.

\subsection{The trivial group}\label{subsec:trivial}

Lorenzini \cite[Corollary~2.10]{L24} has shown that every finite, connected graph admits an arithmetical structure with trivial critical group. In this subsection, we show how to explicitly construct such arithmetical structures on star graphs and complete graphs. 

It is not difficult to obtain the trivial group as a critical group on~$S_n$ for any~$n$ since the arithmetical structure with $r_i=1$ for all $i\in\{0,1,\dotsc,n\}$ has trivial critical group. One cannot obtain a corresponding arithmetical structure on the complete graph~$K_n$ with $n\geq2$ vertices via the star-clique operation in this case since $d_0=n\geq 2$. However, we can construct an arithmetical structure with trivial critical group on~$S_n$ that does have $d_0=1$, and applying the star-clique operation then gives an arithmetical structure with trivial critical group on~$K_n$.

Notice that if $\vdhat = (d_1,\dotsc,d_{n-1},d_n)$ determines an arithmetical structure on~$S_n$, that is, if each~$d_i$ is a positive integer and $d_0\coloneqq\sum_{i=1}^n\frac{1}{d_i}\in\mathbb{Z}$, then, using $\frac{1}{d_n}=\frac{1}{d_n+1}+\frac{1}{d_n(d_n+1)}$, we get that $\widehat{\vd'}=(d_1,\dotsc,d_{n-1},d_n+1,d_n(d_n+1))$ determines an arithmetical structure on $S_{n+1}$ with the same value of~$d_0$. Moreover, if $r_0=d_n$, we have that $r_0'=d_n(d_n+1)$. In this case, if $\CK$ (resp.~$\CK'$) is the critical group of the arithmetical structure determined by~$\vdhat$ (resp.~$\widehat{\vd'}$), Theorem~\ref{thm:criticalgroupstar} implies that
\[\CK \oplus (\ZZ/d_n\ZZ)^2 \cong \bigoplus_{i=1}^n \ZZ/d_i\ZZ\]
and
\[\CK' \oplus (\ZZ/d_n(d_n+1)\ZZ)^2 \cong \Biggl(\bigoplus_{i=1}^{n-1} \ZZ/d_i\ZZ\Biggr)\oplus \ZZ/(d_n+1)\ZZ \oplus \ZZ/d_n(d_n+1)\ZZ.\]
Since $\gcd(d_n, d_n+1)=1$, we can then conclude that $\CK' \cong \CK$. Note that, while the above construction itself does not rely on the fact that $d_n$ is the least common multiple of the other~$d_i$, the conclusion $\CK' \cong \CK$ does use this. 

\begin{example}\label{ex:sylvester}
By starting with the arithmetical structure on~$S_2$ given by $\vdhat=(2,2)$ and repeatedly applying this construction, we can construct arithmetical structures with trivial critical groups and with $d_0=1$ on~$S_n$ for any $n \geq 2$. The first few are listed below. 
\begin{center}
\begin{tabular}{lll}
$n$ & $\vd$ & $\vr$ \\
2 & $(2, 2, 1)$ & $(1,1, 2)$ \\
3 & $(2, 3, 6, 1)$ & $(3, 2, 1, 6)$ \\
4 & $(2,3,7,42, 1)$ & $(21,14,6,1, 42)$ \\
5 & $(2,3,7,43,1806, 1)$ & $(903,602,258,42,1, 1806)$ \\
6 & $(2,3,7,43,1807,3263442, 1)$ & $(1631721,1087814,466208,75894,1806,1, 3263442)$ \\
\end{tabular}
\end{center}
It follows from Curtiss~\cite{Curtiss} that, for each~$n$, these examples maximize the largest entry of a vector~$\vd$ for an arithmetical structure on~$S_n$. Because each of these structures has $d_0=1$, they will translate (via Corollary~\ref{cor:criticalgroupcomplete}) to arithmetical structures with trivial critical groups on the complete graph~$K_n$.
\end{example}

We note that the $\vdhat$ in the above table can also be thought of as being obtained by letting the first $n-1$ terms come from the beginning of Sylvester's sequence \cite[A000058]{OEIS} and then setting $d_n=\prod_{i=1}^{n-1}d_i$. This sequence, which grows doubly exponentially, is given by $s_{n+1} = s_n^2 - s_n + 1$, or alternatively by $s_n=\prod_{i=1}^{n-1} s_i +1$, with $s_1=2$. The first few terms are
\[2, 3, 7, 43, 1807, 3263443, 10650056950807, 113423713055421844361000443.\]
We will reference this sequence again later in this section.

More generally, we will obtain an arithmetical structure on~$S_n$ with trivial critical group for any set of relatively prime integers $\{d_i\}_{i=1}^{n-1}$ for which $\sum_{i=1}^{n-1} \frac{1}{d_i}+\frac{1}{\prod_{i=1}^{n-1} d_i} =1$. Sets of these numbers are well studied; see, for example, \cite{Anne, ACF, BD, BH}. However, this is not the only way to obtain arithmetical structures with trivial critical groups; $\vdhat=(2,3,10,15)$ gives another example.

\subsection{Cyclic groups}\label{subsec:cyclic}

We can generalize the construction given in Section~\ref{subsec:trivial} to obtain certain cyclic groups as critical groups. We first introduce the following operation. Suppose $\vdhat=(d_1,\dotsc, d_{n-1},d_n)$ determines an arithmetical structure on~$S_n$. For any integer $a\mid d_n$, the vector $\CD_a\bigl(\vdhat\bigr)=(d_1,\dotsc,d_{n-1},d_n+a,d_n(d_n+a)/a)$ determines an arithmetical structure on $S_{n+1}$ with the same $d_0$ value because $\frac{1}{d_n}=\frac{1}{d_n+a}+\frac{1}{{d_n}(d_n+a)/a}$.

We now state the following theorem.

\begin{theorem}\label{thm:expand}
Suppose $\vdhat=(d_1,\dotsc, d_{n-1},d_n)$ determines an arithmetical structure on~$S_n$ with critical group~$\CK$. For any integer $a\mid d_n$ with $\gcd(r_0/d_n, d_n/a+1)=1$, the critical group associated to $\CD_a\bigl(\vdhat\bigr)$ is given by $\CK'\cong \CK \oplus \ZZ/a\ZZ$.
\end{theorem}

\begin{proof}
By Theorem~\ref{thm:criticalgroupstar}, we know that 
\[\CK \oplus(\ZZ/r_0\ZZ)^2 \cong \bigoplus_{i=1}^n \mathbb{Z}/d_i\mathbb{Z}\]
and 
\[\CK' \oplus(\ZZ/r_0'\ZZ)^2 \cong \Biggl(\bigoplus_{i=1}^{n-1} \mathbb{Z}/d_i\mathbb{Z}\Biggr) \oplus \ZZ/(d_n+a)\ZZ \oplus \ZZ\big/(d_n(d_n+a)/a)\ZZ.\]
This implies that we will have $\CK'\cong \CK\oplus\ZZ/a\ZZ$ exactly if
\[(\ZZ/r_0\ZZ)^2 \oplus\ZZ/(d_n+a)\ZZ \oplus \ZZ/(d_n(d_n/a+1))\ZZ\cong (\ZZ/r_0'\ZZ)^2 \oplus \ZZ/d_n\ZZ \oplus \ZZ/a\ZZ.\]

Since $r_0$ is equal to $\lcm(d_1,d_2, \dotsc, d_{n-1})$ by Lemma~\ref{lem:primes}, we have that 
\[\begin{split}
r_0'&=\lcm(d_1,d_2, \dotsc,d_{n-1}, d_n+a, d_n(d_n+a)/a)\\
&=\lcm(r_0, d_n+a, d_n(d_n+a)/a)\\
&= \lcm(r_0, d_n(d_n+a)/a)\\
&=d_n\lcm(r_0/d_n, d_n/a+1)\\
& = r_0(d_n/a+1),
\end{split}\]
where the last equality holds by the assumption that $\gcd(r_0/d_n, d_n/a+1)=1$. Therefore we need to show that
\begin{equation}\label{eq:CKandCKprime}
(\ZZ/r_0\ZZ)^2 \oplus\ZZ/(d_n+a)\ZZ \oplus \ZZ/(d_n(d_n/a+1))\ZZ\cong(\ZZ/r_0(d_n/a+1)\ZZ)^2 \oplus \ZZ/d_n\ZZ \oplus \ZZ/a\ZZ.
\end{equation}

Let $b=\gcd(r_0,d_n/a+1)$. Since we have assumed that $\gcd(r_0/d_n, d_n/a+1)=1$, we also have $\gcd(d_n,d_n/a+1)=b,$ and in particular $b\mid a$. We consider the primary factors associated to both sides of~\eqref{eq:CKandCKprime} for each prime $p\mid r_0'=r_0(d_n/a+1)$. Let $p^\alpha$ be the largest power of~$p$ dividing $r_0$, $p^\beta$ be the largest power of~$p$ dividing $d_n$, $p^\gamma$ be the largest power of~$p$ dividing $a$, and $p^\delta$ be the largest power of~$p$ dividing $d_n/a+1$. The primary factors associated to~$p$ of the left side of~\eqref{eq:CKandCKprime} are given by
\[ (\ZZ/p^\alpha\ZZ)^2 \oplus \ZZ/p^{\gamma+\delta}\ZZ \oplus \ZZ/p^{\beta+\delta}\ZZ,\]
and those factors of the right side of~\eqref{eq:CKandCKprime} are given by 
\[(\ZZ/p^{\alpha+\delta}\ZZ)^2 \oplus \ZZ/p^\beta\ZZ \oplus \ZZ/p^\gamma\ZZ.\]

We now show that the primary factors of the sides of~\eqref{eq:CKandCKprime} agree by checking cases.
\begin{itemize}
\item If $p\mid b$, then $p$ also divides $r_0$, $d_n$, $a$, and $d_n/a+1$. Since $\gcd(r_0/d_n, d_n/a+1)=1$, it must be the case that $\alpha=\beta$. Furthermore, since $p\mid (d_n/a+1)$ and $p\mid d_n$, we must have that $\beta=\gamma$. Therefore the primary factors of associated to~$p$ of both sides of~\eqref{eq:CKandCKprime} are $(\ZZ/p^\alpha\ZZ)^2 \oplus(\ZZ/p^{\alpha+\delta}\ZZ)^2$.
\item If $p\mid d_n$ but $p\nmid b$, then $p\nmid (d_n/a+1)$. Therefore $\delta=0$ and the primary factors associated to~$p$ of both sides of~\eqref{eq:CKandCKprime} are $(\ZZ/p^\alpha\ZZ)^2 \oplus \ZZ/p^\beta\ZZ \oplus \ZZ/p^\gamma\ZZ$.
\item If $p\mid r_0$ but $p\nmid d_n$, then $p\mid (r_0/d_n)$. Since $\gcd(r_0/d_n,d_n/a+1)=1$, this means $p\nmid (d_n/a+1)$. Therefore we have $\beta=\gamma=\delta=0$ and thus get $(\ZZ/p^\alpha\ZZ)^2$ as the primary factors of both sides of~\eqref{eq:CKandCKprime}.
\item If $p\nmid r_0$ and $p\mid (d_n/a+1)$, then $\alpha=\beta=\gamma=0$ and we get $(\ZZ/p^\delta\ZZ)^2$ as the primary factors of both sides of~\eqref{eq:CKandCKprime}.\qedhere 
\end{itemize}
\end{proof}

Beginning with an arithmetical structure that has a trivial critical group and satisfies the conditions of the theorem for some $a$ and~$d_n$, we can obtain a structure with a cyclic critical group $\ZZ/a\ZZ$. For example, the critical group associated to $\vdhat=(2,3,11,15,110)$ is trivial and we can take $d_5=110$ and $a=5$. Since $r_0=330$ for this example, we can compute that $\gcd(330/110,110/5+1)=\gcd(3,23)=1$, and thus $\CD_5\bigl(\vdhat\bigr) = (2,3,11,15,115, 2530)$ has critical group $\ZZ/5\ZZ$.

We state the following immediate corollary of Theorem~\ref{thm:expand}. 

\begin{corollary}\label{cor:expand}
Suppose $\vdhat=(d_1,\dotsc, d_{n-1},d_n)$ determines an arithmetical structure on~$S_n$ with $r_0=d_n$ and critical group~$\CK$. For any integer~$a$ that divides $d_n$, the vector $\CD_a\bigl(\vdhat\bigr)$ determines an arithmetical structure on $S_{n+1}$ with critical group $\CK'\cong \CK \oplus \ZZ/a\ZZ$ and $\lcm\bigl(\CD_a\bigl(\vdhat\bigr)\bigr)=d_n(d_n+a)/a$. 
\end{corollary}

\begin{proof}
If $r_0=d_n$, then $r_0/d_n=1$, so the hypotheses of Theorem~\ref{thm:expand} are trivially satisfied. Therefore the critical group of the structure determined by $\CD_a\bigl(\vdhat\bigr)$ is $\CK'\cong \CK \oplus \ZZ/a\ZZ$. Also, $\lcm\bigl(\CD_a\bigl(\vdhat\bigr)\bigr)=\lcm(d_n, d_n+a, d_n(d_n+a)/a) = d_n(d_n+a)/a$.
\end{proof}

Using the construction from Section~\ref{subsec:trivial}, we obtain the following proposition.

\begin{proposition}\label{prop:Sylvesterprimes}
Consider the set of Sylvester primes $\mathcal{SP}$ \cite[A126263]{OEIS}, i.e.\ those primes that divide some number in Sylvester's sequence. For any product $c=\prod p_i$ of distinct primes $p_i\in \mathcal{SP}$, the cyclic group $\ZZ/c\ZZ$ can be realized as a critical group of an arithmetical structure on $S_n$ and~$K_n$ for some~$n$.
\end{proposition}

\begin{proof}
Suppose $p_i$ is a Sylvester prime dividing $s_{m_i}$, the $m_i$-th term of Sylvester's sequence. Consider the arithmetical structure on~$S_n$ from Example~\ref{ex:sylvester} with $n=\max(m_i)+1$. Since $s_{m_i}$ must divide $d_n$ for each~$i$, we have that $c=\prod p_i$ also divides~$d_n$. Since $r_0=d_n$ for this structure and the critical group is trivial, we can use Corollary~\ref{cor:expand} to see that $\CD_c\bigl(\vdhat\bigr)= (d_1, \dotsc, d_{n-1}, d_n+c, d_n(d_n+c)/c)$ is an arithmetical structure on $S_{n+1}$ with critical group $\ZZ/c\ZZ$. We have $d_0=1$ from the construction in Example~\ref{ex:sylvester}, so therefore we can apply the star-clique operation to $\CD_c\bigl(\vdhat\bigr)$ to obtain an arithmetical structure on $K_{n+1}$ with critical group $\ZZ/c\ZZ$.
\end{proof}

For example, $13$ is a Sylvester prime dividing $s_6=3263442$. Therefore, one way to attain $\ZZ/13\ZZ$ as a critical group is to start with the arithmetical structure 
\[\vdhat =(2,3,7,43,1807, 3263443, 10650056950806),\] 
which has trivial critical group, and do the above construction with $a=13$ to get 
\[\CD_{13}\bigl(\vdhat\bigr) =(2,3,7,43,1807, 3263443, 10650056950819, 8724901004273049618800778),\]
which has critical group $\ZZ/13\ZZ$.

Since the terms of Sylvester's sequence are coprime to each other, there are infinitely many Sylvester primes, and thus Proposition~\ref{prop:Sylvesterprimes} shows we can obtain infinitely many cyclic groups as critical groups of arithmetical structures on star graphs.

\subsection{High-rank groups}\label{subsec:highrank}

When looking at which noncyclic groups might appear as critical groups of arithmetical structures on~$S_n$, we can bound the number of generators of a critical group in the following way.

\begin{proposition}\label{prop:components}
Let $n\geq2$. Given an arithmetical structure on $S_n$ or~$K_n$, the invariant factor decomposition of the corresponding critical group can have at most $n-2$ components. Moreover, this bound is sharp.
\end{proposition}

\begin{proof}
For star graphs, the proof of Theorem~\ref{thm:criticalgroupstar} shows that $D_1(B(\vd))=D_2(B(\vd))=1$. This implies that $\alpha_1=\alpha_2=1$, which means the invariant factor decomposition of the critical group has at most $n-2$ components. Because all arithmetical structures on~$K_n$ correspond to arithmetical structures on~$S_n$ with the same critical group, the same result holds for~$K_n$.

We can show that this bound is sharp by considering the Laplacian arithmetical structure $\vr=\mathbf{1}$ on~$K_n$, which is well known to have critical group $(\ZZ/n\ZZ)^{n-2}$. The corresponding arithmetical structure on~$S_n$, determined by $\vdhat = (n,n,\dotsc,n)$, also has critical group $(\ZZ/n\ZZ)^{n-2}$.
\end{proof}

To obtain more examples of higher-rank critical groups, we consider the $\CD_a$~construction of the previous subsection and note that if $a\mid d_n$ then $a\mid d_n(d_n+a)/a$. Therefore we can inductively use this construction to add multiple copies of $\ZZ/a\ZZ$ to a critical group. As an example, starting with the arithmetical structure on~$S_2$ determined by $\vdhat=(2,2)$ and repeatedly applying the construction with $a\in\{1,2\}$ we can obtain arithmetical structures with critical group $(\ZZ/2\ZZ)^m$ on~$S_n$ for each $m\in\{0,1,\dotsc,n-2\}$. Because each of these structures has $d_0=1$, this result also translates to~$K_n$.

For example, if we want to attain $(\ZZ/2\ZZ)^{n-2}$ on~$S_n$ for each~$n$, we would use $a=2$ repeatedly to obtain the following sequence of arithmetical structures.
\begin{center}
    \begin{tabular}{lll}
    $n$ & $\vdhat$ & $\CK$ \\
    2 & $(2,2)$ & trivial group \\
    3 & $(2,4,4)$ & $\ZZ/2\ZZ$ \\
    4 & $(2,4,6,12)$ & $(\ZZ/2\ZZ)^2$ \\
    5 & $(2,4,6,14,84)$ & $(\ZZ/2\ZZ)^3$ \\
    6 & $(2,4,6,14,86,3612)$ & $(\ZZ/2\ZZ)^4$
    \end{tabular}
\end{center}

This construction generalizes and allows us to construct many more examples of critical groups associated to star and complete graphs. For example, starting with $\vdhat=(3,3,3)$ we can obtain the following critical groups on~$S_n$:
\begin{itemize}
    \item $(\ZZ/3\ZZ)^m$ for all $m\in[n-2]$, and
    \item $(\ZZ/3\ZZ)^{m_3} \oplus (\ZZ/2\ZZ)^{m_2}$ for all $m_3\in \{2,3,\dotsc,n-2\}$ and $m_2\in\{0,1,\dotsc,n-4\}$. 
\end{itemize}

From $\vdhat=(2,5,5,10)$, we can get $(\ZZ/2\ZZ)^{m_2} \oplus (\ZZ/5\ZZ)^{m_5}$ for all $m_2\in\{0,1,\dotsc,n-4\}$ and $m_5\in\{1,2,\dotsc,n-3\}$. 

Recall that for a finite abelian group $\bigoplus_{k=1}^t \ZZ/m_k\ZZ$ the \emph{exponent} of the group is defined to be $\lcm(m_1,m_2,\dotsc,m_{t})$. As a full generalization of the ideas of this section, we obtain the following result.

\begin{proposition}\label{prop:allgroups}
Suppose $G$ is a finite abelian group of exponent~$e$. If there exists a set $\{d_i\}_{i=1}^n$ of pairwise relatively prime integers such that $\sum_{i=1}^{n} \frac{1}{d_i}+\frac{1}{\prod_{i=1}^{n} d_i} =1$ and $e$ divides $\prod_{i=1}^{n} d_i$, then there is an arithmetical structure on a star graph with critical group~$G$.
\end{proposition}

\begin{proof}
It follows from the hypotheses that $\vdhat = (d_1,d_2,\dotsc,d_n,\prod d_i)$ defines an arithmetical structure on $S_{n+1}$ with trivial critical group.

Let $G$ be a finite abelian group whose exponent is $e$, so that in particular we can write $G \cong \bigoplus_{k=1}^t \ZZ/a_k\ZZ$, where $a_{t}\mid e$ and $a_k\mid a_{k+1}$ for all $k\in[t-1]$. We recall from Corollary~\ref{cor:expand} that $\CD_{a_{t}}\bigl(\vdhat\bigr)$ gives an arithmetical structure on $S_{n+2}$ with critical group $\ZZ/a_{t}\ZZ$. Moreover, $\CD_{a_{t}}\bigl(\vdhat\bigr)$ will be a vector whose last entry is the least common multiple of the other entries. We can use the operation repeatedly and obtain that 
\[\CD_{a_{1}}\bigl(\CD_{a_{2}}\bigl(\dotsb \bigl(\CD_{a_{t}}\bigl(\vdhat\bigr)\bigr)\dotsb\bigr)\bigr)\] 
gives an arithmetical structure on $S_{n+t+1}$ with critical group~$G$.
\end{proof}

As an example, consider $\vdhat=(2,3,11,23,31,47058)$. The group 
\[H=(\ZZ/2\ZZ)^3\oplus(\ZZ/11\ZZ)^2\oplus\ZZ/31\ZZ\] 
has exponent $2\cdot11\cdot31=682\mid47058$, and we can write $H=\ZZ/2\ZZ\oplus \ZZ/22\ZZ\oplus \ZZ/682\ZZ $. Therefore, we can compute 
\[\begin{split}
\CD_{2}\bigl(\CD_{22}\bigl(\CD_{682}\bigl(\vdhat\bigr)\bigr)\bigr) & = \CD_{2}(\CD_{22}(\CD_{682}(2,3,11,23,31,47058)))\\
& =\CD_{2}(\CD_{22}(2,3,11,23,31,47740,3294060))\\
& =\CD_{2}(2,3,11,23,31,47740,3294082,493222897860),
\end{split}\]
which equals
\[(2,3,11,23,31,47740,3294082,493222897862,121634413487201219187660).\]
Thus we have constructed an arithmetical structure with critical group~$H$.

It was conjectured in~\cite{BH} (and was originally stated as a question by Erd\H{o}s) that for any set of pairwise relatively prime integers $\{d_i\}_{i=1}^{k}$ with $\sum_{i=1}^{k} \frac{1}{d_i}<1$ there exists an integer $m>k$ and integers $\{d_i\}_{i=k+1}^{m}$ such that the $\{d_i\}_{i=1}^{m}$ are pairwise relatively prime and $ \sum_{i=1}^{m} \frac{1}{d_i}+\frac{1}{\prod_{i=1}^{m} d_i} =1$. The $k=1$ case of this conjecture combined with Proposition~\ref{prop:allgroups} would answer Question~\ref{ques:allgroups} in the affirmative. 

\subsection{Products of groups}\label{subsec:products}

Let $\vd'$ be an arithmetical structure on~$S_n$ with critical group~$\CK'$ and $\vd''$ an arithmetical structure on~$S_m$ with critical group~$\CK''$. Consider the structure obtained by taking $\vdhat$ to be the concatenation of $\widehat{\vd'}$ and~$\widehat{\vd''}$. It is clear that the sum of the reciprocals of the entries of~$\vdhat$ will be an integer (and that the value of~$d_0$ will be the sum of $d_0'$ and~$d_0''$) and therefore that $\vdhat$ defines an arithmetical structure on $S_{m+n}$. Moreover, from Theorem~\ref{thm:criticalgroupstar} we have that 
\[\begin{split}
\CK \oplus (\ZZ/r_0\ZZ)^2 & \cong \bigoplus_{i=1}^{n+m} \ZZ/d_i\ZZ \\
& \cong \Biggl(\bigoplus_{i=1}^n \ZZ/d_i'\ZZ\Biggr) \oplus \Biggl(\bigoplus_{i=1}^m \ZZ/d_i''\ZZ\Biggr) \\
& \cong \CK' \oplus (\ZZ/r_0'\ZZ)^2 \oplus \CK'' \oplus (\ZZ/r_0''\ZZ)^2.
\end{split}\]

The set of entries of~$\vdhat$ is the union of the sets of entries in $\widehat{\vd'}$ and~$\widehat{\vd''}$, so therefore $r_0=\lcm(r_0',r_0'')$. Since for any two positive integers $a$ and~$b$ we have that $\ZZ/a\ZZ \oplus \ZZ/b\ZZ \cong (\ZZ/\lcm(a,b)\ZZ) \oplus (\ZZ/\gcd(a,b)\ZZ)$, it follows that
\[\CK \cong \CK' \oplus \CK'' \oplus (\ZZ/\gcd(r_0',r_0'')\ZZ)^2.\]
Noting that the least common multiple of the $\vd$-values on the leaves is the $\vr$-value on the central vertex, this implies that if we have two arithmetical structures on $S_m$ and~$S_n$ whose $\vr$-values at their respective central vertices are relatively prime then we can obtain a new arithmetical structure on $S_{m+n}$ whose critical group is the direct sum of the individual critical groups. However, this new structure will have a $\vd$-value greater than~$1$ at the central vertex and therefore these constructions do not translate to complete graphs.

For example, if we take the structure ${\vd'}=(3,3,7,7,21,1)$ with $\vr'=(7,7,3,3,1,21)$ and the structure ${\vd''}=(2,5,5,10,1)$ with $\vr''=(5,2,2,1,10)$, we have that $\CK'=\ZZ/21\ZZ$ and $\CK''=\ZZ/5\ZZ$. Consider the structure given by $\vd=(3,3,7,7,21,2,5,5,10,2)$ with $\vr=(70,70,30,30,10,105,42,42,21,210)$. It follows from the above argument that the critical group of this structure is $\CK\cong \CK'\oplus\CK''\cong \ZZ/105\ZZ$. 

\begin{remark}
Not every arithmetical structure on~$S_n$ with $d_0\geq2$ comes from two or more smaller structures in the manner described above. For example, the structure with $\vdhat=(2,3,3,5,5,5,5,30)$ on~$S_8$ has $d_0=2$, but there is no way to realize this as a concatenation of two other $\vdhat$-structures on smaller star graphs.
\end{remark}

\subsection{Every group is the subgroup of a critical group} \label{subsec:subgroup}

In this subsection, we show that while we cannot yet prove that every finite abelian group appears as a critical group of an arithmetical structure on a star graph, it is the case that we can obtain every group as a subgroup.

\begin{proposition}\label{prop:subgroup}
Given any finite abelian group~$G$, there is some~$n$ and some arithmetical structure $(\vd,\vr)$ on~$S_n$ such that $G\leq \mathcal{K}(S_n; \vd, \vr)$.
\end{proposition}

\begin{proof}
Let $G$ be a finite abelian group, let $\bigoplus_{i=1}^m \mathbb{Z}/p_i^{e_i}\mathbb{Z}$ be its primary decomposition, and let $r_0=\lcm(p_1^{e_1}, p_2^{e_2}, \dotsc, p_m^{e_m})$. Note that there are some nonnegative integers $k$ and~$\ell$ such that 
\[kr_0 = 2+\ell+\sum_{i=1}^m \frac{r_0}{p_i^{e_i}}.\]
Indeed, one can take 
\[k=\biggl\lceil \biggl(2+\sum_{i=1}^m \frac{r_0}{p_i^{e_i}}\biggr)\Bigm/r_0\biggr\rceil \quad\text{ and }\quad \ell=kr_0-\sum_{i=1}^m \frac{r_0}{p_i^{e_i}}-2.\]
In this case, we can take $n=m+2+\ell$ and let the arithmetical structure be given by $\vd$-values $d_i = {p_i}^{e_i}$ for $i \in [m]$, $d_{i}=r_0$ for $i\in\{m+1,m+2,\dotsc,n\}$, and $d_0=k$, and $\vr$-values $r_0=\lcm(p_1^{e_1}, p_2^{e_2}, \dotsc, p_m^{e_m})$, $r_i=r_0/p_i^{e_i}$ for $i \in [m]$, and $r_{i}=1$ for $i\in\{m+1,m+2,\dotsc,n\}$. In this case, by Proposition~\ref{prop:CriticalGroupStarWithTwo1}, the critical group is $\mathcal{K}(S_n; \vd, \vr)\cong G \oplus (\mathbb{Z}/r_0\mathbb{Z})^\ell$.
\end{proof}

The proof of this proposition shows that something stronger is actually true. For any finite abelian group~$G$, there is some other finite abelian group~$H$ such that $G\oplus H$ appears as a critical group of some arithmetical structure on a star graph.

In general, the $\ell$ from the proof of Proposition~\ref{prop:subgroup} can be large. For example, if we use this approach to find an arithmetical structure on a star graph whose critical group contains $G=(\mathbb{Z}/10\mathbb{Z})^2\oplus\mathbb{Z}/25\mathbb{Z}\oplus \mathbb{Z}/3\mathbb{Z}$, we would take $n=68$ together with $\vr=(15,15,6,50,1,1,\dotsc,1,150)$ and $\vd=(10,10,25,3,150,150,\dotsc,150,1)$. In this case, the critical group would be 
\[K(S_{68};\vd,\vr) \cong (\mathbb{Z}/10\mathbb{Z})^2\oplus\mathbb{Z}/25\mathbb{Z}\oplus \mathbb{Z}/3\mathbb{Z} \oplus (\mathbb{Z}/150\mathbb{Z})^{62}.\]

\begin{proposition}\label{prop:maked1}
Suppose $(\vd,\vr)$ is an arithmetical structure on~$S_n$ for $n\geq1$. Define $(\vd',\vr')$ by setting $d'_i=d_0d_i$ and $r'_i=r_i$ for all $i\in[n]$, $d'_0=1$, and $r'_0=d_0r_0$. We have that $(\vd',\vr')$ is an arithmetical structure on~$S_n$ and $\mathcal{K}(S_n; \vd, \vr)\leq\mathcal{K}(S_n; \vd', \vr')$.
\end{proposition}

\begin{proof}
It is straightforward to check that $(\vd',\vr')$ is an arithmetical structure on~$S_n$.

Let $\alpha_k$ be the invariant factors of $\mathcal{K}(S_n; \vd, \vr)$, and let $\alpha'_k$ be the invariant factors of $\mathcal{K}(S_n; \vd', \vr')$. Since $d'_i=d_0d_i$ for all $i\in[n]$, we have that $D_k(B(\vd'))=d_0^{k-2}D_k(B(\vd))$ for all $k\in\{3,4,\dotsc,n\}$ and $D_1(B(\vd'))=D_2(B(\vd'))=1$. If follows that $\alpha'_k=d_0\alpha_k$ for all $k\in\{3,4,\dotsc,n\}$ and that $\alpha'_1=\alpha_1=1$ and $\alpha'_2=\alpha_2=1$. Therefore $\alpha_k$ divides $\alpha'_k$ for all $k\in[n]$, so $\mathcal{K}(S_n; \vd, \vr)\leq\mathcal{K}(S_n; \vd', \vr')$.
\end{proof}

For example, suppose $\vd=(1, 2, 2, 3, 3, 6, 6, 3)$ and $\vr=(6, 3, 3, 2, 2, 1, 1, 6)$. In this case, the critical group is 
\[\CK(S_7;\vd,\vr)\cong(\ZZ/1\ZZ)\oplus(\ZZ/2\ZZ)^2\oplus(\ZZ/3\ZZ)^2 \cong (\ZZ/1\ZZ)^3\oplus(\ZZ/6\ZZ)^2.\] 
If we set $r_0'=3\cdot 6=18$, then we get $\vd'=(3,6,6,9,9,18,18,1)$ and 
\[\CK(S_7,\vd',\vr') \cong (\ZZ/3\ZZ)\oplus (\ZZ/6\ZZ)^2 \oplus (\ZZ/9\ZZ)^2 \cong (\ZZ/3\ZZ)^3\oplus (\ZZ/18\ZZ)^2.\]

\begin{proposition}\label{prop:Ksubgroup}
Given any finite abelian group~$G$, there is some~$n$ and some arithmetical structure $(\vd,\vr)$ on~$K_n$ such that $G\leq \mathcal{K}(K_n; \vd, \vr)$.
\end{proposition}

\begin{proof}
First apply Proposition~\ref{prop:subgroup} to get an arithmetical structure $(\vd,\vr)$ on~$S_n$ for which $G\leq \mathcal{K}(S_n; \vd, \vr)$. Then apply Proposition~\ref{prop:maked1} to get $\mathcal{K}(S_n; \vd, \vr)\leq\mathcal{K}(S_n; \vd', \vr')$ for an arithmetical structure $(\vd',\vr')$ on~$S_n$ with $d'_0=1$. Finally use the star-clique operation to get an arithmetical structure $(\vd'',\vr'')$ on~$K_n$ with $\mathcal{K}(K_n; \vd'', \vr'')\cong\mathcal{K}(S_n; \vd', \vr')$. We then have that $G\leq\mathcal{K}(K_n; \vd'', \vr'')$.
\end{proof}

\subsection{Largest critical group}\label{subsec:order}

Let us consider the following question.

\begin{question}\label{ques:maxorder}
What is the largest order of a critical group of an arithmetical structure on~$S_n$ \textup{(}and on~$K_n$\textup{)}?
\end{question}

For $n \in \{2,3,4,5,6,7\}$, a brute force computation finds the largest order of a critical group of an arithmetical structure on~$S_n$ (and on~$K_n$), shown in the following table. 

\begin{center}
\begin{tabular}{c|c|c|c|c|c|c}
$n$ & $2$ & $3$ & $4$ & $5$ & $6$ & $7$ \\ \hline
largest order & $1$ & $3$ & $16$ & $128$ & $5292$ & $9784908$ 
\end{tabular}
\end{center}

Let $a_n$ be the sequence defined by the recursion $a_n=a_{n-1}^2 + a_{n-1}$ with $a_1=1$. This sequence is described at \cite[A007018]{OEIS}. The first few terms are 
\[1,2, 6, 42, 1806, 3263442, 10650056950806, 113423713055421844361000442,\dotsc.\]
Note that this sequence is exactly $\{s_{n}-1\}$, where $s_n$ is the $n$-th term of Sylvester's sequence. Equivalently, $a_n=\prod_{i=1}^{n-1} s_i$. 

\begin{theorem}
If $\CK$ is a critical group of an arithmetical structure on~$S_n$ \textup{(}or~$K_n$\textup{)} for $n\geq2$, then $\abs{\CK}<\frac{n!}{2}a_{n-2}^2$.
\end{theorem}

\begin{proof}
If $d_0>1$, then by Proposition~\ref{prop:maked1} there is an arithmetical structure with $d_0=1$ that has a larger order critical group. Therefore the largest order critical group associated to~$S_n$ will come from an arithmetical structure with $d_0=1$ and will also be the largest order critical group associated to~$K_n$. 

Assuming $d_1\leq d_2\leq\dotsb\leq d_n$, we must have $d_i\leq (n-i+1)a_{i}$ since the sum of the first $i-1$ unit fractions cannot be closer than $1/a_i$ to the next integer (as shown in~\cite{Curtiss}) and otherwise we would not be able to reach the next integer. Since $\abs{\CK}=\frac{\prod_{i=1}^n d_i}{r_0^2}$ and $d_n\leq r_0$, we have that 
\[\abs{\CK} \leq\prod_{i=1}^{n-2}((n-i+1)a_i)=\frac{n!}{2}\prod_{i=1}^{n-2}a_i.\]
Notice that $\prod_{i=1}^{n-2}a_i=a_{n-2}\prod_{i=1}^{n-3}a_i< a_{n-2}^2$. Therefore, we have that $\abs{\CK}<\frac{n!}{2}a_{n-2}^2$.
\end{proof}

The sequence~$a_n$ grows doubly exponentially (much faster than~$n!$), and we see from the next example that the largest order of a critical group does indeed grow doubly exponentially with respect to~$n$. 

Consider the arithmetical structure on~$S_n$ with 
\[\vdhat=(a_1+1,a_2+1,\dotsc,a_{n-3}+1,3a_{n-2},3a_{n-2},3a_{n-2}).\]
The arithmetical structure given by $\widehat{\vd'}=(a_1+1, a_2+1, \dotsc, a_{n-3}+1, a_{n-2})$ is exactly the structure on $S_{n-2}$ from Example~\ref{ex:sylvester}, which means that $\prod_{i=1}^{n-3}(a_i+1)=a_{n-2}$. Since $\frac{1}{a_{n-2}}=\frac{1}{3a_{n-2}}+\frac{1}{3a_{n-2}}+\frac{1}{3a_{n-2}}$, we have that $\vdhat$ gives a valid structure on~$S_n$. Since $r_0=3a_{n-2}$, we know from Theorem~\ref{thm:criticalgroupstar} that the order of the critical group associated to~$\vdhat$ is 
\[\abs{\CK}=\frac{\prod_{i=1}^n d_i}{r_0^{2}} = \frac{a_{n-2}(3a_{n-2})^3}{(3a_{n-2})^2} = 3a_{n-2}^2.\]
We conjecture that this is the largest order critical group associated to~$S_n$ for sufficiently large~$n$.

\begin{conjecture}
For $n\geq 6$, the largest order of a critical group of an arithmetical structure on~$S_n$ \textup{(}or on~$K_n$\textup{)} is $3a_{n-2}^2$.
\end{conjecture}

Consider the arithmetical structure on~$S_n$ determined by 
\[\vdhat=(a_1+1,a_2+1,\dotsc,a_{n-2}+1,2a_{n-1},2a_{n-1}).\]
This gives an arithmetical structure since it is obtained from the structure on $S_{n-1}$ in Example~\ref{ex:sylvester} by $\frac{1}{a_{n-1}}=\frac{1}{2a_{n-1}}+\frac{1}{2a_{n-1}}$. Using Theorem~\ref{thm:criticalgroupstar} and that $a_{n-1} = \prod_{i=1}^{n-2}(a_{i}+1)$ (with coprime factors), we can see that the critical group is the cyclic group $\ZZ/a_{n-1}\ZZ$. We conjecture that this is the largest cyclic critical group associated to~$S_n$.

\begin{conjecture}
For $n\geq 4$, the largest cyclic group that can be realized as a critical group of an arithmetical structure on~$S_n$ \textup{(}or on~$K_n$\textup{)} is $\ZZ/a_{n-1}\ZZ$.
\end{conjecture}

\subsection{Number of critical groups associated to \texorpdfstring{$S_n$}{S\_n}}\label{subsec:number}

Let us now consider the following question.

\begin{question}
For a given $n\geq 2$, how many distinct abelian groups \textup{(}up to isomorphism\textup{)} appear as critical groups associated to~$S_n$? 
\end{question}

Let $\CG(G)$ be the set of critical groups associated to the graph~$G$. Using the clique-star operation, we know that $\CG(K_n)\subseteq \CG(S_n)$. For $n\in\{2,3,4,5,6,7\}$, brute force computation tells us that $\CG(S_n)=\CG(K_n)$ as sets; the cardinality of these sets is given below.
\begin{center}
    \begin{tabular}{c|c|c|c|c|c|c}
    $n$ & $2$ & $3$ & $4$ & $5$ & $6$ & $7$ \\ \hline
    $\abs{\CG(S_n)}$ & $1$ & $3$ & $10$ & $56$ & $574$ & $20420$
    \end{tabular}
\end{center}

Clearly, $\abs{\CG(S_n)}$ is bounded above by the number of arithmetical structures on~$S_n$ (\cite[A156871]{OEIS}) and $\abs{\CG(K_n)}$ is bounded above by the number of arithmetical structures on~$K_n$ (\cite[A002966]{OEIS}), both of which grow doubly exponentially. We show an exponential lower bound for $\abs{\CG(S_n)}$ below.

\begin{proposition}
For $n\geq 2$, we have $\abs{\CG(S_n)}\geq 2^{n-2}$.
\end{proposition}

\begin{proof}
First note that $\abs{\CG(S_2)}=1$. We will show that, for each $n\geq 2$, we have $\abs{\CG(S_{n+1})}\geq 2\abs{\CG(S_n)}$. 

Notice that each critical group realized on~$S_n$ is also realized on $S_{n+1}$. Specifically, if $\vdhat=(d_1,d_2, \dotsc ,d_n)$ determines an arithmetical structure on~$S_n$ with critical group~$\CK$, then $\widehat{\vd'}=(1,d_1,d_2,\dotsc, d_n)$ determines an arithmetical structure on $S_{n+1}$ with the same critical group~$\CK$. By Proposition~\ref{prop:components}, these critical groups have at most $n-2$ invariant factors.

Now, let us see that we can also obtain at least $\abs{\CG(S_n)}$ critical groups associated to $S_{n+1}$ with $n-1$ invariant factors. Consider any arithmetical structure $(\vd,\vr)$ on~$S_n$ with critical group $\CK\cong \bigoplus_{k=3}^n \ZZ/\alpha_{k}\ZZ$. By Theorem~\ref{thm:criticalgroupstar}, it must be the case that
\[\CK \oplus(\ZZ/r_0\ZZ)^2 \cong \Biggl(\bigoplus_{k=3}^n \ZZ/\alpha_{k}\ZZ\Biggr) \oplus (\ZZ/r_0\ZZ)^2 \cong \bigoplus_{i=1}^n \ZZ/d_i\ZZ.\]
We have that $\widehat{\vd'}=(d_0+1, d_1(d_0+1), d_2(d_0+1), \dotsc, d_n(d_0+1))$ also determines an arithmetical structure on $S_{n+1}$. Note that $r_0'=r_0(d_0+1)$. By Theorem~\ref{thm:criticalgroupstar}, this structure has critical group~$\CK'$, where
\[\CK' \oplus(\ZZ/r_0(d_0+1)\ZZ)^2 \cong \ZZ/(d_0+1)\ZZ\oplus\Biggl(\bigoplus_{i=1}^n \ZZ/d_i(d_0+1)\ZZ\Biggr).\]
As
\[\bigoplus_{i=1}^n \ZZ/d_i(d_0+1)\ZZ\cong\Biggl(\bigoplus_{k=3}^n \ZZ/\alpha_{k}(d_0+1)\ZZ\Biggr)\oplus (\ZZ/r_0(d_0+1)\ZZ)^2,\]
this implies that the invariant factor decomposition of~$\CK'$ is
\[\ZZ/(d_0+1)\ZZ \oplus \Biggl(\bigoplus_{k=3}^n \ZZ/\alpha_{k}(d_0+1)\ZZ\Biggr).\]

Since groups in the first collection have a different number of invariant factors from those in the second collection, we have that $\abs{\CG(S_{n+1})}\geq 2\abs{\CG(S_n)}$, as desired.
\end{proof}

\section*{Acknowledgments}

We would like to thank ICERM for their support through the Collaborate@ICERM program. Alexander Diaz-Lopez’s research is supported by National Science Foundation grant DMS-2211379. Joel Louwsma was partially supported by a Niagara University Summer Research Award.

\bibliographystyle{amsplain}
\bibliography{StarAndComplete}

\newpage

\appendix
\section{Critical groups associated to small star graphs}

Here we list all groups that appear as critical groups on~$S_n$ for $n\in\{2,3,4,5,6\}$. For each of these values of~$n$, the same groups that appear as critical groups associated to~$S_n$ also appear as critical groups associated to~$K_n$.

\subsection*{Critical groups associated to \texorpdfstring{$S_2$}{S\_2}} 

The $1$~group that occurs is the trivial group.

\subsection*{Critical groups associated to \texorpdfstring{$S_3$}{S\_3}} 

The $3$~groups that occur are the trivial group, $\ZZ/2\ZZ$, and $\ZZ/3\ZZ$.

\subsection*{Critical groups associated to \texorpdfstring{$S_4$}{S\_4}} 

The $10$ groups that occur are: 
\begin{itemize}
    \item Cyclic groups, $\ZZ/m\ZZ$, for~$m$ in the list: 
    \begin{tabular}{lllll} 
    $1$, & $2$, & $3$, & $5$, & $6$.
    \end{tabular}
    \item Groups with two invariant factors, 
    \[\ZZ/m_1\ZZ\oplus\ZZ/m_2\ZZ,\] 
    for $(m_1,m_2)$ in the list: 
    \begin{longtable}{lllll} 
    $(2,2)$, & $(2,4)$, & $(2,6)$, & $(3,3)$, & $(4,4)$.
    \end{longtable}
\end{itemize}

\subsection*{Critical groups associated to \texorpdfstring{$S_5$}{S\_5}} 

The $56$ groups that occur are: 
\begin{itemize}
    \item Cyclic groups, $\ZZ/m\ZZ$, for~$m$ in the list: 
\begin{longtable}{lllllllllllllll}
$1$, & $2$, & $3$, & $5$, & $6$, & $7$, & $10$, & $11$, & $12$, & $13$, & $14$, & $15$, & $18$, & $21$, & $42$.
\end{longtable}
    \item Groups with two invariant factors, 
    \[\ZZ/m_1\ZZ\oplus\ZZ/m_2\ZZ,\]
    for $(m_1,m_2)$ in the list: 
    \begin{longtable}{lllllllll} 
$(2, 2)$, & 
$(2, 4)$, &
$(2, 6)$, &
$(2,8)$, &
$(2,10)$, &
$(2, 12)$, &
$(2, 14)$, &
$(2, 20)$, &
$(2, 24)$, \\
$(2, 3)$, &
$(3, 3)$, &
$(3, 6)$, &
$(3, 9)$, &
$(3, 12)$, &
$(3, 15)$, &
$(3, 18)$, &
$(4,4)$, &
$(4, 12)$, \\ 
$(5,5)$, &
$(5,10)$, &
$(6, 6)$, &
$(6, 12)$, &
$(6, 18)$, &
$(7, 7)$, &
$(10,10)$.
    \end{longtable}
     \item Groups with three invariant factors,
     \[\ZZ/m_1\ZZ\oplus\ZZ/m_2\ZZ\oplus\ZZ/m_3\ZZ,\] 
     for $(m_1,m_2,m_3)$ in the list: 
    \begin{longtable}{llllllll} 
$(2, 2,2)$, & 
$(2, 2,4)$, &
$(2,2, 6)$, &
$(2,2,10)$, &
$(2,2,12)$, &
$(2, 4,4)$, &
$(2, 4,8)$, &
$(2, 4,12)$, \\
$(2, 6,6)$, &
$(2, 8,8)$, &
$(3, 3,3)$, &
$(3, 3, 6)$, &
$(3, 3, 9)$, &
$(3, 6,6)$, &
$(4,4,4)$, &
$(5,5,5)$.
    \end{longtable}
\end{itemize}

\subsection*{Critical groups associated to \texorpdfstring{$S_6$}{S\_6}} 

The $574$ groups that occur are: 
\begin{itemize}
\item Cyclic groups $\ZZ/m\ZZ$ for~$m$ in the list: 
\begin{longtable}{lllllllllllll}
$1$, & $2$, & $3$, & $4$, & $5$, & $6$, & $7$, & $9$, & $10$, & $11$, & $12$, & $13$, & $14$, \\ $15$, & $17$, & $18$, & $20$, & $21$, & $22$, & $23$, & $24$, & $25$, & $26$, & $28$, & $30$, & $31$, \\ $33$, & $34$, & $35$, & $38$, & $39$, & $41$, & $42$, & $43$, & $45$, & $46$, & $50$, & $51$, & $54$, \\ $55$, & $57$, & $58$, & $59$, & $60$, & $65$, & $66$, & $70$, & $75$, & $77$, & $78$, & $82$, & $84$, \\ $85$, & $86$, & $87$, & $90$, & $93$, & $94$, & $95$, & $98$, & $102$, & $105$, & $106$, & $110$, & $111$, \\ $114$, & $120$, & $129$, & $130$, & $133$, & $138$, & $141$, & $150$, & $154$, & $156$, & $161$, & $203$, & $210$, \\ $231$, & $238$, & $255$, & $258$, & $301$, & $329$, & $399$, & $462$, & $525$, & $546$, & $602$, & $903$, & $1806$.
\end{longtable}

\item Groups with two invariant factors, 
\[\ZZ/m_1\ZZ\oplus\ZZ/m_2\ZZ,\] 
for $(m_1,m_2)$ in the list: 
\begin{longtable}{llllllll}
$(2, 2)$, & $(2, 4)$, & $(2, 6)$, & $(2, 8)$, & $(2, 10)$, & $(2, 12)$, & $(2, 14)$, & $(2, 16)$, \\
$(2, 18)$, & $(2, 20)$, & $(2, 22)$, & $(2, 24)$, & $(2, 26)$, & $(2, 28)$, & $(2, 30)$, & $(2, 34)$, \\
$(2, 36)$, & $(2, 38)$, & $(2, 40)$, & $(2, 42)$, & $(2, 44)$, & $(2, 46)$, & $(2, 48)$, & $(2, 50)$, \\
$(2, 52)$, & $(2, 54)$, & $(2, 56)$, & $(2, 60)$, & $(2, 62)$, & $(2, 64)$, & $(2, 66)$, & $(2, 68)$, \\
$(2, 70)$, & $(2, 74)$, & $(2, 78)$, & $(2, 80)$, & $(2, 84)$, & $(2, 86)$, & $(2, 88)$, & $(2, 90)$, \\
$(2, 92)$, & $(2, 100)$, & $(2, 102)$, & $(2, 110)$, & $(2, 112)$, & $(2, 114)$, & $(2, 116)$, & $(2, 120)$, \\
$(2, 124)$, & $(2, 132)$, & $(2, 138)$, & $(2, 140)$, & $(2, 150)$, & $(2, 154)$, & $(2, 156)$, & $(2, 168)$, \\
$(2, 170)$, & $(2, 174)$, & $(2, 182)$, & $(2, 200)$, & $(2, 210)$, & $(2, 230)$, & $(2, 300)$, & $(2, 308)$, \\
$(2, 322)$, & $(2, 336)$, & $(2, 350)$, & $(2, 420)$, & $(2, 462)$, & $(2, 600)$, & $(2, 924)$, & $(2, 966)$, \\
$(3, 3)$, & $(3, 6)$, & $(3, 9)$, & $(3, 12)$, & $(3, 15)$, & $(3, 18)$, & $(3, 21)$, & $(3, 24)$, \\
$(3, 27)$, & $(3, 30)$, & $(3, 33)$, & $(3, 39)$, & $(3, 42)$, & $(3, 45)$, & $(3, 51)$, & $(3, 54)$, \\
$(3, 57)$, & $(3, 60)$, & $(3, 63)$, & $(3, 66)$, & $(3, 69)$, & $(3, 75)$, & $(3, 78)$, & $(3, 87)$, \\
$(3, 90)$, & $(3, 99)$, & $(3, 102)$, & $(3, 105)$, & $(3, 114)$, & $(3, 117)$, & $(3, 120)$, & $(3, 126)$, \\
$(3, 156)$, & $(3, 165)$, & $(3, 171)$, & $(3, 210)$, & $(3, 231)$, & $(3, 315)$, & $(3, 342)$, & $(3, 357)$, \\
$(3, 630)$, & $(4, 4)$, & $(4, 12)$, & $(4, 20)$, & $(4, 24)$, & $(4, 28)$, & $(4, 36)$, & $(4, 48)$, \\
$(4, 60)$, & $(4, 84)$, & $(5, 5)$, & $(5, 10)$, & $(5, 15)$, & $(5, 20)$, & $(5, 25)$, & $(5, 30)$, \\
$(5, 35)$, & $(5, 45)$, & $(5, 50)$, & $(5, 55)$, & $(5, 60)$, & $(5, 65)$, & $(5, 75)$, & $(5, 90)$, \\
$(5, 110)$, & $(6, 6)$, & $(6, 12)$, & $(6, 18)$, & $(6, 24)$, & $(6, 30)$, & $(6, 36)$, & $(6, 42)$, \\ 
$(6, 48)$, & $(6, 54)$, & $(6, 60)$, & $(6, 66)$, & $(6, 72)$, & $(6, 78)$, & $(6, 84)$, & $(6, 90)$,\\ 
$(6, 108)$, & $(6, 126)$, & $(6, 132)$, & $(6, 156)$, & $(6, 168)$, & $(6, 180)$, & $(6, 198)$, & $(6, 210)$, \\
$(6, 216)$, & $(6, 336)$, & $(6, 378)$, & $(7, 7)$, & $(7, 14)$, & $(7, 21)$, & $(7, 35)$, & $(7, 42)$, \\
$(7, 49)$, & $(7, 70)$, & $(7, 77)$, & $(7, 91)$, & $(7, 98)$, & $(7, 147)$, & $(7, 294)$, & $(8, 24)$, \\
$(9, 9)$, & $(9, 18)$, & $(10, 10)$, & $(10, 20)$, & $(10, 30)$, & $(10, 40)$, & $(10, 50)$, & $(10, 60)$, \\
$(10, 70)$, & $(10, 100)$, & $(11, 22)$, & $(12, 12)$, & $(12, 24)$, & $(12, 36)$, & $(12, 48)$, & $(12, 60)$, \\ 
$(13, 13)$, & $(13, 26)$, & $(13, 39)$, & $(13, 78)$, & $(14, 14)$, & $(14, 28)$, & $(14, 42)$, & $(14, 56)$, \\ 
$(14, 70)$, & $(14, 84)$, & $(14, 98)$, & $(14, 168)$, & $(14, 210)$, & $(15, 15)$, & $(15, 30)$, & $(15, 45)$, \\
$(17, 17)$, & $(18, 18)$, & $(18, 36)$, & $(19, 19)$, & $(20, 20)$, & $(21, 21)$, & $(21, 42)$, & $(21, 63)$, \\ 
$(21, 105)$, & $(21, 126)$, & $(22, 22)$, & $(22, 66)$, & $(24, 24)$, & $(26, 26)$, & $(30, 30)$, & $(33, 33)$,\\ $(42, 42)$, & $(42, 84)$, & $(42, 126)$.
\end{longtable}

\item Groups with three invariant factors, 
\[\ZZ/m_1\ZZ\oplus\ZZ/m_2\ZZ\oplus\ZZ/m_3\ZZ,\] 
for $(m_1,m_2,m_3)$ in the list: 
\begin{longtable}{lllllll}
$(2, 2, 2)$, & $(2, 2, 4)$, & $(2, 2, 6)$, & $(2, 2, 8)$, & $(2, 2, 10)$, & $(2, 2, 12)$, & $(2, 2, 14)$, \\ $(2, 2, 16)$, & $(2, 2, 18)$, & $(2, 2, 20)$, & $(2, 2, 22)$, & $(2, 2, 24)$, & $(2, 2, 26)$, & $(2, 2, 28)$, \\ $(2, 2, 30)$, & $(2, 2, 34)$, & $(2, 2, 36)$, & $(2, 2, 38)$, & $(2, 2, 40)$, & $(2, 2, 42)$, & $(2, 2, 44)$, \\ $(2, 2, 48)$, & $(2, 2, 50)$, & $(2, 2, 52)$, & $(2, 2, 56)$, & $(2, 2, 60)$, & $(2, 2, 66)$, & $(2, 2, 68)$, \\ $(2, 2, 70)$, & $(2, 2, 76)$, & $(2, 2, 78)$, & $(2, 2, 80)$, & $(2, 2, 84)$, & $(2, 2, 90)$, & $(2, 2, 102)$, \\ $(2, 2, 104)$, & $(2, 2, 110)$, & $(2, 2, 120)$, & $(2, 2, 130)$, & $(2, 2, 132)$, & $(2, 2, 140)$, & $(2, 2, 156)$, \\ $(2, 2, 220)$, & $(2, 2, 240)$, & $(2, 2, 312)$, & $(2, 4, 4)$, & $(2, 4, 8)$, & $(2, 4, 12)$, & $(2, 4, 20)$, \\ $(2, 4, 24)$, & $(2, 4, 28)$, & $(2, 4, 32)$, & $(2, 4, 36)$, & $(2, 4, 40)$, & $(2, 4, 44)$, & $(2, 4, 48)$, \\ $(2, 4, 52)$, & $(2, 4, 56)$, & $(2, 4, 60)$, & $(2, 4, 72)$, & $(2, 4, 80)$, & $(2, 4, 84)$, & $(2, 4, 120)$, \\ $(2, 4, 140)$, & $(2, 4, 168)$, & $(2, 6, 6)$, & $(2, 6, 12)$, & $(2, 6, 18)$, & $(2, 6, 24)$, & $(2, 6, 30)$, \\ $(2, 6, 36)$, & $(2, 6, 42)$, & $(2, 6, 54)$, & $(2, 6, 60)$, & $(2, 6, 84)$, & $(2, 6, 90)$, & $(2, 6, 120)$, \\ $(2, 8, 8)$, & $(2, 8, 16)$, & $(2, 8, 24)$, & $(2, 8, 32)$, & $(2, 8, 40)$, & $(2, 8, 48)$, & $(2, 8, 56)$, \\ $(2, 8, 96)$, & $(2, 8, 120)$, & $(2, 10, 10)$, & $(2, 10, 20)$, & $(2, 10, 30)$, & $(2, 10, 50)$, & $(2, 10, 60)$, \\ $(2, 12, 12)$, & $(2, 12, 24)$, & $(2, 12, 36)$, & $(2, 12, 48)$, & $(2, 12, 60)$, & $(2, 12, 72)$, & $(2, 14, 14)$, \\ $(2, 14, 28)$, & $(2, 14, 42)$, & $(2, 16, 16)$, & $(2, 16, 48)$, & $(2, 18, 18)$, & $(2, 20, 20)$, & $(2, 20, 40)$, \\ $(2, 20, 60)$, & $(2, 24, 24)$, & $(2, 24, 48)$, & $(2, 24, 72)$, & $(2, 28, 28)$, & $(2, 30, 30)$, & $(3, 3, 3)$, \\ $(3, 3, 6)$, & $(3, 3, 9)$, & $(3, 3, 12)$, & $(3, 3, 15)$, & $(3, 3, 18)$, & $(3, 3, 21)$, & $(3, 3, 27)$, \\ $(3, 3, 30)$, & $(3, 3, 33)$, & $(3, 3, 39)$, & $(3, 3, 42)$, & $(3, 3, 45)$, & $(3, 3, 54)$, & $(3, 3, 60)$, \\ $(3, 3, 63)$, & $(3, 3, 126)$, & $(3, 6, 6)$, & $(3, 6, 12)$, & $(3, 6, 18)$, & $(3, 6, 24)$, & $(3, 6, 30)$, \\ $(3, 6, 36)$, & $(3, 6, 42)$, & $(3, 6, 60)$, & $(3, 6, 72)$, & $(3, 6, 90)$, & $(3, 9, 9)$, & $(3, 9, 18)$, \\ $(3, 9, 27)$, & $(3, 9, 45)$, & $(3, 9, 54)$, & $(3, 12, 12)$, & $(3, 12, 24)$, & $(3, 12, 36)$, & $(3, 15, 15)$, \\ $(3, 15, 30)$, & $(3, 18, 18)$, & $(3, 18, 36)$, & $(3, 18, 54)$, & $(3, 21, 21)$, & $(3, 30, 30)$, & $(4, 4, 4)$, \\ $(4, 4, 8)$, & $(4, 4, 12)$, & $(4, 4, 20)$, & $(4, 4, 24)$, & $(4, 8, 8)$, & $(4, 12, 12)$, & $(5, 5, 5)$, \\ $(5, 5, 10)$, & $(5, 5, 15)$, & $(5, 5, 25)$, & $(5, 5, 30)$, & $(5, 10, 10)$, & $(5, 10, 20)$, & $(5, 10, 30)$, \\ $(6, 6, 6)$, & $(6, 6, 12)$, & $(6, 6, 18)$, & $(6, 6, 30)$, & $(6, 6, 36)$, & $(6, 12, 12)$, & $(6, 12, 24)$, \\ $(6, 12, 36)$, & $(6, 18, 18)$, & $(6, 24, 24)$, & $(7, 7, 7)$, & $(7, 7, 14)$, & $(7, 14, 14)$, & $(9, 9, 9)$, \\ $(10, 10, 10)$.
\end{longtable}

\item Groups with four invariant factors 
\[\ZZ/m_1\ZZ\oplus\ZZ/m_2\ZZ\oplus\ZZ/m_3\ZZ\oplus\ZZ/m_4\ZZ,\] 
for $(m_1,m_2,m_3,m_4)$ in the list: 
\begin{longtable}{llllll}
$(2, 2, 2, 2)$, & $(2, 2, 2, 4)$, & $(2, 2, 2, 6)$, & $(2, 2, 2, 10)$, & $(2, 2, 2, 12)$, & $(2, 2, 2, 14)$, \\ $(2, 2, 2, 20)$, & $(2, 2, 2, 22)$, & $(2, 2, 2, 24)$, & $(2, 2, 2, 26)$, & $(2, 2, 2, 28)$, & $(2, 2, 2, 30)$, \\ $(2, 2, 2, 36)$, & $(2, 2, 2, 42)$, & $(2, 2, 2, 84)$, & $(2, 2, 4, 4)$, & $(2, 2, 4, 8)$, & $(2, 2, 4, 12)$, \\ $(2, 2, 4, 16)$, & $(2, 2, 4, 20)$, & $(2, 2, 4, 24)$, & $(2, 2, 4, 28)$, & $(2, 2, 4, 40)$, & $(2, 2, 4, 48)$, \\ $(2, 2, 4, 60)$, & $(2, 2, 6, 6)$, & $(2, 2, 6, 12)$, & $(2, 2, 6, 18)$, & $(2, 2, 6, 24)$, & $(2, 2, 6, 30)$, \\ $(2, 2, 6, 36)$, & $(2, 2, 8, 8)$, & $(2, 2, 8, 24)$, & $(2, 2, 10, 10)$, & $(2, 2, 10, 20)$, & $(2, 2, 12, 12)$, \\ $(2, 2, 12, 24)$, & $(2, 2, 12, 36)$, & $(2, 2, 14, 14)$, & $(2, 2, 20, 20)$, & $(2, 4, 4, 4)$, & $(2, 4, 4, 8)$, \\ $(2, 4, 4, 12)$, & $(2, 4, 4, 20)$, & $(2, 4, 4, 24)$, & $(2, 4, 8, 8)$, & $(2, 4, 8, 16)$, & $(2, 4, 8, 24)$, \\ $(2, 4, 12, 12)$, & $(2, 4, 16, 16)$, & $(2, 6, 6, 6)$, & $(2, 6, 6, 12)$, & $(2, 6, 6, 18)$, & $(2, 6, 12, 12)$, \\ $(2, 8, 8, 8)$, & $(2, 10, 10, 10)$, & $(3, 3, 3, 3)$, & $(3, 3, 3, 6)$, & $(3, 3, 3, 9)$, & $(3, 3, 3, 15)$, \\ $(3, 3, 3, 18)$, & $(3, 3, 6, 6)$, & $(3, 3, 6, 12)$, & $(3, 3, 6, 18)$, & $(3, 3, 9, 9)$, & $(3, 3, 12, 12)$, \\ $(3, 6, 6, 6)$, & $(4, 4, 4, 4)$, & $(4, 4, 4, 8)$, & $(4, 4, 4, 12)$, & $(4, 4, 8, 8)$, & $(5, 5, 5, 5)$, \\ $(6, 6, 6, 6)$.
\end{longtable}
\end{itemize}

\end{document}